\documentclass[10pt, a4paper]{article}
\usepackage{geometry}
\usepackage[english]{babel}

\usepackage{amsmath,amssymb,amsthm,amsthm}

\usepackage{mathrsfs}

\usepackage{todonotes}  

\usepackage[noadjust]{cite} 
\usepackage{xcolor}

\usepackage{hyperref}
\hypersetup{
    colorlinks,
    citecolor=black,
    filecolor=black,
    linkcolor=black,
    urlcolor=black    
}
\usepackage[utf8]{inputenc}
\usepackage{amsmath}
\usepackage{amssymb}
\usepackage{mathtools}
\usepackage{cancel}
\usepackage {graphicx}
\usepackage{amsmath}
\usepackage{floatrow}
\usepackage{multirow}
\usepackage[english]{babel}
\usepackage{mathalfa}
\usepackage{dutchcal}
\usepackage{amsmath}
\usepackage{tikz}

\usepackage{amssymb,amsmath,amsthm}
\usepackage{graphicx}

\usepackage{amsfonts}
\usepackage{latexsym}
\usepackage{color}
\usepackage[english]{babel}
\usepackage{graphicx} 
\usepackage{subcaption}
\usepackage{caption}
\usepackage{mathrsfs}
\usepackage{multicol}
\usepackage{fullpage}
\usepackage{xcolor}
\usepackage{hyperref}
\hypersetup{
   colorlinks=true,
  linkcolor=blue,
  filecolor=magenta,      
 urlcolor=cyan,
}
\hypersetup{
    colorlinks=true,
    linkcolor=blue,
    filecolor=magenta,      
    urlcolor=cyan,
}
\usepackage{bbold}

\newtheorem{teorema}{Theorem}[section]
\newtheorem{definizione}[teorema]{Definition}
 \newtheorem{corollario}[teorema]{Corollary}
\newtheorem{lemma}[teorema]{Lemma}
\newtheorem{proposizione}[teorema]{Proposition}
\newtheorem{osservazione}[teorema]{Remark}

\usepackage{mathtools}
\newcommand{\numberset}{\mathbb}

\newcommand{\R}{\numberset{R}}

\newcommand{\KEYWORDS}[1]{
  \noindent\textbf{Keywords:} #1
}

\newcommand{\AMSSUBJ}[1]{
  \noindent\textbf{AMS Subject Classification:} #1
}

\title{Random rotational invariance of integration by parts formulas within a Bismut-type approach}
\date{}

\author{
Susanna Dehò \footnote{Dipartimento di Matematica, Universit\`a degli Studi di Milano, \emph{susanna.deho@unimi.it}} \and
	Francesco C. De Vecchi\footnote{Dipartimento di Matematica ``Felice Casorati'', Universit\`a degli Studi di Pavia
		\emph{francescocarlo.devecchi@unipv.it}}
	\and 
	Paola Morando \footnote{DISAA, Universit\`a degli Studi di Milano,  \emph{paola.morando@unimi.it}}
	\and
	Stefania Ugolini \footnote{Dipartimento di Matematica, Universit\`a degli Studi di Milano, \emph{stefania.ugolini@unimi.it}}
}

\begin{document}
\maketitle

\begin{abstract} 
	The stochastic rotational invariance of an integration by parts formula inspired by the Bismut approach to Malliavin calculus is proved in the framework of the Lie symmetry theory of stochastic differential equations. The non-trivial effect of the rotational invariance of the driving Brownian motion in the derivation of the integration by parts formula is discussed and the invariance property of the formula is shown via applications to some explicit two-dimensional Brownian motion-driven stochastic models.
\end{abstract}

\\
\AMSSUBJ{60H10; 58D19; 60H07}

\section{Introduction}
The study of the symmetries and invariance properties of stochastic processes is important to understand their long-term behavior, simplify their mathematical analysis, identify conservation laws, and improve numerical methods
(see, e.g., 
\cite{grandemostro,romano,privault2019invariance,Thieullen1997Symmetries}).\\

More specifically, the study of invariance properties of stochastic differential equations (SDEs) via a Lie symmetry approach plays a fundamental role in both the theoretical understanding and in the practical analysis of their solution processes.
A Lie symmetry analysis of SDEs, including both deterministic and random symmetries (\cite{albeverio2019some,gaetaspadaro}), 
can be useful for reducing symmetric equations to simpler ones (see, e.g., \cite{paperriduzione1,paperriduzione,lazaro2009reduction}), for finding martingales associated with the solutions to the equation, which are a generalization of conservation laws of deterministic ODEs (see \cite{noether,Thieullen1997Symmetries}), and for investigating the relation with the Lie symmetries of the associated Fokker-Planck or Kolmogorov equations (\cite{paper2020,gaetarodriguez1}).\\

On the other hand, the integration by parts formula is a cornerstone in the analysis of SDEs and stochastic processes. Its influence extends broadly across probability theory, analysis, quantum field theory, and mathematical finance. The significance of this formula is particularly evident in at least two principal research fields.
First, the integration by parts formula is a key tool for proving the existence and smoothness of densities for the distributions that are solution to SDEs (see, e.g., \cite{nualart}).  In particular, if a process $X_t$ solution to an SDE admits a probability density,
the law (probability distribution) of $X_t$ is not just a discrete or singular measure, but actually has a probability density function  with respect to the Lebesgue measure. 
Smoothness of densities refers to the fact that this probability density function not only exists but is also differentiable (possibly infinitely many times). 

Moreover, integration by parts formulas are at the heart of Malliavin calculus, also known as the stochastic calculus of variations, which provides a probabilistic approach to differentiability on the Wiener space. This framework is essential for sensitivity analysis, hypoellipticity, and to establish integration by parts formulas in infinite-dimensional settings (see, e.g., \cite{bismut,malliavin,nualart}).
In mathematical finance, integration by parts formulas are used to compute Greeks (sensitivities of option prices to parameters) and to develop efficient Monte Carlo methods for derivative pricing (\cite{fournie1999applications,malliavin2006stochastic,privault2007integration}).\\

The Lie symmetry theory applied to the study of SDEs has recently allowed the formulation of a new approach to the Bismut-Malliavin calculus. More precisely,  using the invariance property of the law of the solution process along a given symmetry of the SDE under study, in \cite{paper2023} an analogue of the Bismut-Malliavin quasi-invariance principle was demonstrated. In the paper cited, the derivation of a related integration by parts formula involves symmetries that include spatial diffeomorphisms, deterministic time changes, and probability measure changes via the Girsanov theorem, but leave the Brownian term unchanged.\\

In this paper we investigate the stochastic invariance under rotations of the integration by parts formula introduced in \cite{paper2023} in the framework of Lie symmetry theory for SDEs.
In particular, we introduce a new type of transformation (the random rotation of the original Brownian motion $W$) given by
\begin{equation*} 
dW^\prime_t=B(X_t,t)dW_t,
\end{equation*}
where $B:M\times \R_+ \rightarrow SO(m),$ with $M \subseteq \mathbb{R}^n $ and $SO(m)$ the group of special orthogonal matrices.
This type of transformation was previously introduced in \cite{paper2015} from a slightly different perspective. Indeed, while \cite{paper2015} examines the entire pair $(X,W)$, our focus here is on the diffusion process $X$, motivated by the well-known fact that an Itô diffusion $ X$ can correspond to distinct SDEs that are related by random rotations of the driving Brownian motion.

The inclusion of the Brownian motion rotations among the class of admissible random transformations of an SDE 
of the type
\[
dX_t=\mu(X_t,t)dt+\sigma(X_t,t)dW_t,
\]
implies that the law of a solution $X$ to  $SDE_{\mu,\sigma}$ is no longer identified by coefficients $ (\mu,\sigma)$ but by the whole family $$ ({\mu,\sigma B^{-1}})_{B \in SO(m)}$$ 
that we denote by $\mathcal{Gauge}( SDE_{\mu,\sigma})$.

In other words, if we are interested only in weak solutions to the given SDE, i.e. in diffusion problems, $X$ is no more characterized by the particular $ SDE_{\mu,\sigma}$, but by its associated martingale problem $ (\mu, \sigma\sigma^T)$, i.e. by its infinitesimal generator $$ L=\mu(x,t)\partial_x + \sigma(x,t) \sigma^T(x,t) \partial_{xx}$$ which is as well invariant under the action of $ SO(m)$ (Section 2.2.2). \\

For these reasons, following \cite{paper2015,paper2020,paper2023}, we introduce the concepts of \textit{strong} and \textit{weak} stochastic transformations to mirror the usual distinction between weak and strong solutions to an SDE. 
In particular, strong stochastic transformations leave the Brownian motion invariant 

as it is given \textit{a priori} with the filtered probability space, while  weak ones modify also the Brownian component.

Therefore, in the present paper, we propose a new notion of invariance for SDEs that we call $\mathcal{G}-$invariance.
  
An SDE is $\mathcal{G}-$invariant under the action of a transformation $ T$ if $ T$ preserves the set of weak solutions to the whole family $\mathcal{Gauge}(SDE)$. The related symmetry notion can be characterized in terms of a finite transformation of the SDE coefficients (Theorem \ref{finite_transformation_coefficient}) and by deriving  the new determining equations corresponding to the more powerful infinitesimal transformations (Theorem \ref{new-determining-equation}).

The new concept of invariance can be considered the prototypical example of the more general notion of gauge invariance group proposed in 
 \cite{grandemostro} for semimartingales with jumps.
 Starting from this generalized notion of invariance for SDEs, in this paper we derive a new quasi invariance principle 

as well as the associated integration by parts formula. Furthermore, the $\mathcal{G}-$weak symmetry represents the most general notion of symmetry for which these last two main results just mentioned still hold.   \\

The main result of this paper is the proof of the rotational invariance of the integration by parts formula originally derived in \cite{paper2023}, which reflects the well-known invariance of the infinitesimal generator of an 
SDE under the same transformation (see Theorem \ref{L invariant}).
This is due to the fact that the formula relies on the invariance of the solution law, and this law is determined by its generator $L$. Therefore, a rotation of the driving Brownian motion, which does not affect $L$, preserves both the solution law and the structure of the integration by parts formula.
Although we prove that, despite the rotation of the driving Brownian motion, the integration by parts formula remains unchanged, the formulation of the necessary analytical conditions to rigorously derive the formula requires the discussion of new not trivial terms.\\

Moreover, the inclusion of random rotations of the Brownian motion allows us to provide a new and comprehensive study of invariance properties of the integration by parts formula and their implications for stochastic investigation of the solution to SDEs. For example, in this paper we apply our symmetry approach to the two-dimensional Brownian motion, i.e. to the two-dimensional $ SDE_{0,1} $, obtaining an infinite-dimensional family of symmetries $V_\beta$ (see \eqref{Vb} in Section \ref{Examples}) where $\beta(t)$ is an arbitrary smooth function of time.\\
The symmetry $V_{\beta}$ provides a generalization of the classical Proposition \ref{invarianza per rotazioni}, which ensures that, given $\ B(t)$, a rotation matrix depending on time $\ t$, then  $ W'_t=\int_0^t B(s) dW_s$ is still a $\mathbb{P}-$Brownian motion. This generalization can be obtained by taking advantage of the Girsanov theorem. In particular, considering the process $\tilde{W}_t= B(t)  W_t$, since $W'_{t'}$ is already a $\mathbb{P}$-Brownian motion, $\tilde{W}$ cannot be himself a $\mathbb{P}$-Brownian motion. Nevertheless, since $V_{\beta}$ is a symmetry, by the associated Girsanov transformation we get that $\tilde{W}$ is a $\mathbb{Q}$-Brownian motion, where the probability measure $\mathbb{Q}$ has density with respect to $\mathbb{P}$ given by 
\small $$ \frac{d\mathbb{Q}}{d\mathbb{P}}_{|_{\mathcal{F_T}}}= \exp\Bigg( - \int_0^T B'(s) \cdot W_s dW_s - \frac{1}{2} \int_0^T | B'(s) \cdot W_s|^2 ds \Bigg),$$ \normalsize 
where $B'(t)$ is the matrix obtained by differentiating the entries of the matrix B with respect to time.

Symmetry $\ V_{\beta}$ encodes this invariance property, ensuring that rotating the Brownian motion via$\ B(t)$ yields a solution to the same original SDE, but with respect to the new probability measure $\mathbb{Q}$ (see Remark \ref{invariance mb} for further details). In this sense, it becomes more evident how the theory of symmetries applied to SDEs could be a useful tool for studying invariance properties of stochastic processes and for analyzing their laws. Our general integration by parts formula (see Theorem \ref{teo integrazione per parti 2}) in the particular case of the two-dimensional Brownian motion and along the symmetry $V_{\beta}$ assumes the following form: \begin{multline*}
    0 = \mathbb{E}_{\mathbb{P}} \Big[ F(X_t,Y_t) \int_0^t - Y_s \beta'(s)dW^1_s + X_s \beta'(s) d W^2_s  \Big] \\
    + \mathbb{E}_{\mathbb{P}}\Big[\beta(t) Y_t\partial_x F(X_t,Y_t)- \beta(t)X_t \partial_y F(X_t,Y_t) \Big],
\end{multline*}
where $F$ is a bounded functional with bounded first derivative of the process under study (see Section \ref{Examples}). Furthermore, with appropriate choices of function $\beta(t)$, one can recover from this integration by parts formula well known results in probability theory, such as Stein's lemma and Isserli's theorem (see Remark \ref{remark stein} for further details).\\

The generality of our approach is illustrated through applications to some stochastic models driven by Brownian motion. More specifically we analyze in details the integration by parts formulas which can be rigorously derived for the two-dimensional Brownian motion, an additive perturbation of Brownian motion and the stochastic Lotka-Volterra model.\\

The paper is organized as follows.
In Section \ref{Finite stochastic transformations and symmetries} we illustrate the random rotation of the driving Brownian motion and the related invariance property both of the Brownian motion itself and of the associated SDE.
The set of finite stochastic transformations for an SDE is recalled in Section \ref{Finite_stochastic_transformations_SDE}, together with a detailed description of their algebraic and geometrical structure.
In Section \ref{InvarianceProperties} a novel notion of invariance of SDEs is proposed by introducing the concept of symmetry as a transformation which preserves the law of the solution process.
A more general quasi-invariance principle and all the analytical results, preparatory to the integration by parts formula, are derived and explained in Section \ref{quasi invariance principle}.
In Section \ref{Rotational invariante formula} the rotational invariant integration by parts formula is obtained and three notable examples are carefully discussed in Section \ref{Examples}.\\

Throughout this paper, the Einstein summation convention for repeated indices will be used.

\section{SDE invariance under random rotation of the driving Brownian motion}\label{Finite stochastic transformations and symmetries}
In this section we specify the notion of weak solution to an SDE according to the martingale problem formulation \cite{stroock_varadhan_1979} and  we introduce the random rotation of the driving Brownian motion as proposed in some recent papers (\cite{paper2015},\cite{paper2020}).
We finally discuss the fact that the random rotation of Brownian motion can be viewed as a particular case of a more general notion of symmetry, called gauge symmetry, introduced in \cite{albeverio}.
\subsection{It\^o diffusions }
We provide a brief overview of the martingale problem formulation associated with an Itô-type SDE and its connection to the concepts of strong and weak solutions.
Let us denote with $M^p[0,T]$ the space of the equivalence classes of the real processes $Y$ defined in some filtered probability space $(\Omega, \mathcal{F}, \mathcal{F}_t, \mathbb{P}),$ taking values on an open subspace $M \subseteq \R^n$, progressively measurable and such that
\[
\mathbb{E}\Bigg[\int_0^T|Y_s|^p ds\Bigg]<\infty.
\]
Speaking of equivalence classes, we mean that we identify two processes $Y$ and $Y^\prime$ if $Y_s=Y^\prime_s$ for almost every $s$, with probability one.

\begin{definizione}[Weak and strong solutions]\label{weak and strong solutions}
   Let us fix a finite time horizon $\ T >0.$  Given $\ s \rightarrow \mu(X_s,s)=(\mu_i(X_s,s))_{i=1,...,n}$ and $\ s \rightarrow \sigma(X_s,s)=(\sigma_{ij}(X_s,s))_{i=1,...,n \ \ j=1,...,m}$ processes in $\ M^1[0,T]$, $\ M^2[0,T]$ respectively, given $\ x \in \mathbb{R}^n,$ we say that the $ SDE_{\mu,\sigma}$
    \begin{equation}\label{SDE}
        dX_t=\mu(X_t,t)dt + \sigma(X_t,t) dW_t, \quad X_0=x
    \end{equation}
    admits $ X$ as weak solution  if  there exist a filtered probability space  $\ (\Omega, \mathcal{F}, \mathcal{F}_t, \mathbb{P})$ and an $\mathcal{F}_t$-Brownian motion $ W$ taking values in $\mathbb{R}^m$ such that, for any $ t \in [0,T],$
    \begin{equation}\label{sol}  
     X^i_t-x^i=\int_0^t \mu_i(X_s,s)ds + \int_0^t \sigma_{ij}(X_s,s) dW^j_s.
    \end{equation}
    In this case we say that $ X$ is a weak solution to $SDE_{\mu,\sigma}$ driven by $ W$. \\
Differently, we say that \eqref{SDE} admits $X$ as  strong solution if \eqref{sol} is satisfied for every filtered probability space $(\Omega, \mathcal{F},\mathcal{F}_t,\mathbb{P})$ and for every $\mathcal{F}_t-$Brownian motion $W$.
\end{definizione}
\begin{osservazione}
If $X$ is a strong solution, the filtered probability space and the version of the Brownian motion are given a priori, and the solution $X_t$ is then constructed accordingly, ensuring that it is adapted to the filtration generated by the given Brownian motion. On the other hand, if $ X$ is a weak solution, only the functions $\mu$ and $\sigma$ are specified in advance, and we seek a pair $ (X,W)$ along with a filtered probability space such that $ X$ is adapted to the filtration and 
$W$ is a Brownian motion w.r.t. it. In this case, $ X$ does not necessarily have to be adapted to the filtration generated by $ W.$ A strong solution is of course also a weak solution, but the converse is not true in general (see \textit{Tanaka equation} in \cite{oksendal}).
\end{osservazione}
We introduce the definition of the \textit{martingale problem} associated with an SDE.
\begin{definizione}
Let $\mu,\sigma$ be previsible path functionals and $ x \in \mathbb{R}^n.$ We say that $ X$ is a solution to the martingale problem $\ (\mu,\sigma\sigma^T)$ starting at $ x$ if $P^x(X_0=x)=1$ and for any $ f \in C_0^{\infty}(\mathbb{R}^n \times \mathbb{R}_+)$
\[ f(X_t,t)-f(x,0)-\int_0^t L_t(f(X_s,s))ds\]
is a martingale w.r.t. $\mathcal{H}^n_t=\sigma(X_s, s \leq t)$, where
\[ L_t= \partial_t + L= \partial_t + \mu \partial + \frac{1}{2} \sigma \sigma^T \partial^2.\]
\end{definizione}
\begin{osservazione}
    If $ X$ is a weak solution to SDE \eqref{SDE} driven by $ W$, then, applying \textit{It\^o formula}, $\forall f \in C_0^{\infty}(\mathbb{R}^n\times \mathbb{R}_+ )$
    \[ f(X_t,t)-f(X_0,0)-\int_0^t L_t(f(X_s,s))ds = \int_0^t D(f(X_s,s))^T \sigma(X_s,s) dW_s\]
    is a martingale and hence $ X$ is also a solution to the martingale problem $ (\mu,\sigma\sigma^T)$ starting at $X_0$.  The converse to this observation is an important result in the study of martingale problems for diffusion processes proved by Stroock and Varadhan  (
    \cite{stroock_varadhan_1979}). So, the previous  formulation of the martingale problem is equivalent to the weak solution one.
\end{osservazione}

\begin{osservazione}\label{autonome non autonome}
 In the following, we will deal with functions defined on $ \mathbb{R}^n \times \mathbb{R}_+  $, that is, depending both on the stochastic process $ X_t$ and the time $ t$. It is always possible to reduce this situation to the autonomous case, using a standard trick. Exploiting the isomorphism between $\mathbb{R}^n\times \mathbb{R}$ and $\mathbb{R}^{n+1}$, we can consider only autonomous functions just working in a higher dimension: we can pass from $ f(x,t)$ defined for $ (x,t) \in \mathbb{R}^n\times \mathbb{R}_+ $ to $  f(\tilde{x})$ defined for $\tilde{x}\in \mathbb{R}^{n+1}$.
\end{osservazione}

\subsection{Random rotations of Brownian motion}\label{Random rotation of Brownian motion}
Stochastic transformations composed of triplets, encompassing  time changes, diffeomorphisms, and measure changes, were introduced in \cite{paper2020} and used in \cite{paper2023} to propose a new integration by parts formula inspired by the Bismut approach to Malliavin calculus. Here, we aim to introduce in the same setting a novel type of transformation: the random rotation of Brownian motion.

In particular, since an Itô diffusion $X$ can correspond to distinct SDEs that are related by random rotations of the Brownian motion we need to slightly modify the definition of invariance and symmetry proposed in \cite{paper2015}. These broader definitions will also clarify the main result of this paper: the rotational invariance of the integration by parts formula originally derived in \cite{paper2023}.
\\
\noindent
 Before examining the action of these transformations on an SDE, we recall the important (random) invariance property of the Brownian motion (see \cite{paper2015}).
 
 \begin{proposizione}\label{invarianza per rotazioni}
Let $\tilde M$ be an open subset of $\mathbb{R}^n$ and $M=\tilde M \times \mathbb{R}_+ $. Let $ B: M \rightarrow SO(m)$
be a smooth function and let $ X$ be a weak solution to $SDE_{\mu,\sigma}$ driven by $ W$. Let $W'$ be the solution to
\begin{equation*} 
d W'_t = B(X_t,t) d W_t, 
\end{equation*}
then $ W'$ is still a Brownian motion. \end{proposizione} 
The previous Proposition \ref{invarianza per rotazioni} allows us to give the following definition. 
\begin{definizione}\label{definizione rotazione browniano} Let $W$ be an $ (\Omega, \mathcal{F}, \mathcal{F}_t, \mathbb{P})$ $ m-$dimensional Brownian motion. A random rotation of $ W$ is a smooth function $ B: M  \rightarrow SO(m)$ leading from $ W$ to $ W'$ in such a way that the following equation is satisfied: \begin{equation*} d W'_t= B(X_t,t) dW_t. \end{equation*} \end{definizione}
Finally, we need to investigate how the coefficients of an SDE change after the action of a transformation of its driving Brownian motion: 
\begin{proposizione}\label{coefficienti mb} 
Let $ B: M  \rightarrow SO(m)$ be a smooth function and let $ X$ be a weak solution to $SDE_{\mu,\sigma}$ driven by $W$. Let $ W'$ be the solution to \begin{equation*} d W'_t = B(X_t,t) d W_t, \end{equation*} then $ X$ is also a weak solution to $ SDE_{\mu',\sigma'}$ driven by $W'$, with 
\begin{equation*} \mu' = \mu; \ \ \ \  \sigma'= \sigma  B^{-1}. 
\end{equation*} 
\end{proposizione}
\begin{proof}
By the invariance property under rotations, $W'$ is a Brownian motion (Proposition \ref{invarianza per rotazioni} ). Moreover, we have that
\begin{equation*} d X_t = \mu dt + \sigma dW_t = \mu' dt + \underbrace{\sigma B^{-1}}_{\sigma'} B dW_t =\mu' dt + \sigma'  \underbrace{B  dW_t}_{dW'_t} = \mu' dt + \sigma' dW'_t, 
\end{equation*} 
\end{proof}

\subsubsection{The prototypical
 example of gauge symmetry}\label{{The prototypical
 example of gauge symmetry}}
The invariance of Brownian motion with respect to (random) rotations has a deep consequence in the definition of Lie symmetries of Brownian motion driven SDEs. Indeed, if $X$ is a diffusion solution to an $ SDE_{\mu,\sigma}$ driven by $ W$ \[ dX_t=\mu(X_t,t)dt+\sigma(X_t,t)dW_t,\]
and $ B: M  \rightarrow SO(m)$ 
 is a predictable process, then $X$ is also a weak solution to the $SDE_{\mu,\sigma B^{-1}}$ driven by $ W_t'=\int_0^t B(X_s,s) dW_s$, i.e.
\[ dX_t=\mu(X_t,t)dt+(\sigma(X_t,t) B^{-1}(X_t,t)) [B(X_t,t) dW_t],\]
 which is a well-defined SDE since, by Proposition \ref{invarianza per rotazioni}, the driving process $[ B(X_t,t) dW_t]$ is still a Brownian motion.\\ 
 The inclusion of the Brownian motion rotations among the class of admissible transformations has as a consequence that the law of a weak solution $X$ to a Brownian motion driven $SDE_{\mu,\sigma}$ is no longer identified by coefficients $ (\mu,\sigma)$ but by the whole family $ ({\mu,\sigma B^{-1}})_{B \in SO(m)}$ related by a rotation of the Brownian component. In other words, if we are interested only in weak solutions to SDEs, i.e. in diffusion problems, $X$ is not characterized by $ SDE_{\mu,\sigma}$, but by its associated martingale problem $ (\mu, \sigma\sigma^T)$, that is, by its infinitesimal generator $ L=\mu\partial + \frac12\sigma \sigma^T \partial^2$.\\
Notice also that, unlike the other transformations studied in \cite{paper2023}, the rotation of the Brownian motion acts only on the Brownian component. Therefore, since we are interested in defining symmetries as transformations that preserve the set of weak solutions to a given SDE, transformations consisting of the rotation of Brownian motion are always symmetries. \\
Indeed, this fact provides the prototypical example of a notion of symmetry for a general class of  SDEs introduced by \cite{albeverio}:
\begin{definizione}
    Let $ Z$ be a semimartingale defined on a Lie group $ N$ with respect to a given filtration $\mathcal{F}_t$. Given a topological group $\mathcal{G}$ with action on $ N$ \[\Xi: \mathcal{G} \times N \rightarrow N,\]
    we say that $Z$ admits $\mathcal{G}$, with action $\Xi_g$ and w.r.t the filtration $\mathcal{F}_t$, as a gauge symmetry group if, for any $\mathcal{F}_t-$predictable locally bounded process $ G_t$ taking values in $\mathcal{G}$, the semimartingale $\tilde{Z}$ solution to the equation $ d\tilde{Z}_t=\Xi_{G_t}(dZ_t)$ has the same law as $Z$.
\end{definizione}
It is easy to verify that, taking $ N=\mathbb{R}^m$ and $Z=W$ in the previous definition, $\mathcal{G}=SO(m)$ is  a gauge symmetry group for Brownian motion.  
 
Moreover, since the set of diffusions related to a given infinitesimal generator $ L$ is preserved under the action of a gauge transformation, a weak symmetry may not transform a given SDE into itself, but can modify the Brownian part of the SDE by performing a random rotation (i.e., a gauge transformation). Thus, it is convenient to introduce the following definition.
\begin{definizione}\label{definizione gauge SDE}
    Given a Brownian motion driven $ SDE_{\mu,\sigma}$, we denote by $\mathcal{ Gauge}(SDE_{\mu,\sigma})$ the set of different  SDEs obtained from $ SDE_{\mu,\sigma}$ applying the action of the Gauge symmetry group $\ SO(m)$, i.e.
    \[ \mathcal{Gauge}(SDE_{\mu,\sigma})=(SDE_{\mu,\sigma B^{-1}})_{B \in SO(m).} \]
\end{definizione}

\begin{osservazione}\label{osservazione gauge SDE}
    We already noticed that if $X$ is a diffusion related to $SDE_{\mu,\sigma}$, then $X$ is also a diffusion related to the whole $\mathcal{Gauge}(SDE_{\mu,\sigma}).$
\end{osservazione}

\subsubsection{Invariance of the infinitesimal generator under random rotations}\label{Invarian under random rotations}
It is well-known that the canonical Markov process associated with a given $SDE_{\mu,\sigma}$, under suitable regularity conditions for the coefficients, is a diffusion with generator 
\begin{equation}\label{L}
 L= \mu_i(x,t) \partial_{x_i} + \frac{1}{2}  a_{ij}(x,t) \partial_{x_i x_j},
\end{equation}
where $ a(x,t)=\sigma(x,t)\sigma(x,t)^T.$ 
Conversely, given a differential operator $ L$ of the form \eqref{L}, where $ a(x,t) \in Mat(n,n)$ is a positive semidefinite matrix, it becomes natural to inquire about the existence of a diffusion process associated with $ L$, and about its uniqueness. Concerning existence, the answer to the problem is positive \textit{provided} there exists a matrix field $\sigma(x,t)$, called square root of $ a$, such that $\sigma(x,t)\sigma(x,t)^T=a(x,t)$ and that $\sigma$ and $\mu$ satisfy the assumptions which ensure the existence of a solution to the SDE (e.g. joint misurability, Lipschitz continuity and sublinear growth in $x$, see,e.g., \cite{baldi}). The answer about uniqueness is more important for the development of the current paper and deserves a more articulated investigation, that we briefly recall from \cite{baldi}. 

\begin{proposizione}\label{L invariant}
    Given a differential operator \eqref{L}
        with $ a(x,t) \in Mat(n,n)$ positive semidefinite, if there exist $\sigma, \tilde{\sigma} \in Mat(n,m)$ such that for every $ (x,t) \in \mathbb{R}^n \times [0,T]$
    \[ \sigma(x,t) \sigma(x,t)^T=a(x,t)=\tilde{\sigma}(x,t) \tilde{\sigma}(x,t)^T,\]
   then there exists an orthogonal matrix $\rho(x,t) \in O(m)$ such that
    \[ \tilde{\sigma}=\sigma \rho.\]
   
\end{proposizione}
\begin{proof}
    See \textit{Lemma 9.4} in \cite{baldi}.
\end{proof}
Let us define a  stochastic process on the trajectory space  $\Omega:=\{\omega: \R^+ \rightarrow M\}, $ and  let us introduce the canonical process  as a coordinate process, i.e. given by $X_t(\omega)=\omega(t)$. In the case of processes which are solutions to  Brownian motion driven SDEs, as in our case, one considers as canonical space $\Omega=\mathcal{C}([0,T], M),$ equipped with the natural filtration.
\begin{corollario}
  Given a differential operator \eqref{L}
        with $ a(x,t) \in Mat(n,n)$ positive semidefinite,  provided there exists $\sigma $ such that $\sigma \sigma^T=a$, with $\mu$ and $\sigma$ satisfying the usual conditions, then on the canonical space $ (\mathcal{C}, \mathcal{M}, (\mathcal{M}_t)_t, (X_t)_t)$ there exists a unique family of probabilities $ (P^x)_x$ such that $(\mathcal{C}, \mathcal{M}, (\mathcal{M}_t)_t, (X_t)_t, (P^x)_x)$\footnote{$(\mathcal{C}, \mathcal{M}, (\mathcal{M}_t)_t, (X_t)_t, (P^x)_x)$ denotes the canonical process, where $X_t$ is the coordinate process on the path space $\mathcal{C}$, equipped with its natural filtration and induced probability measures $(P^x)_x$ for the process starting from $x$. The canonical space provides a general, abstract setting to define stochastic processes and their laws  without reference to a specific probability space. See, e.g, \cite{baldi} for a detailed discussion.} is the realization of a diffusion process associated with $ L$.
\end{corollario}
\begin{proof} See Theorem 9.10 in \cite{baldi}.
\end{proof}
\noindent
Recalling the rotational invariance of Brownian motion, Proposition \ref{L invariant} can be re-stated in the following equivalent way, which is  more useful for our purposes.
\begin{proposizione}\label{Invariance under random rotations of the infinitesimal generator}[Invariance under random rotations of the infinitesimal generator]

Consider two stochastic differential equations $SDE_{\mu,\sigma}$ and $ SDE_{\mu',\sigma'}$ with infinitesimal generator $L$ and  $ L'$, respectively.     The two  SDEs differ by a random rotation of the Brownian part (there exists $ B\in O(m)$ such that $\sigma'=\sigma B^{-1})$ if and only if they share the same infinitesimal generator ($ L=L'$).
\end{proposizione}
\begin{proof}
    If $\sigma'=\sigma B^{-1}$, then, thanks to the orthogonality of $ B$,
    \[L'=\mu_i(x,t) \partial_i + \frac{1}{2} a_{ij}(x,t) \partial_{ij}=\mu_i(x,t) \partial_i + \frac{1}{2} (\sigma' \sigma'^T)_{ij}(x,t) \partial_{ij}=\mu_i(x,t) \partial_i + \frac{1}{2} (\sigma \sigma^T)_{ij}(x,t) \partial_{ij}=L.\]
Conversely, if $ L=L'$ then $\sigma\sigma^T=\sigma'\sigma'^T$ and by  Proposition \ref{L invariant} there exists $\rho\in O(m)$ such that $\sigma=\sigma' \rho$. Set $ B=\rho^{-1}$ and the proof is complete.
\end{proof}
\begin{osservazione}
    Using \textit{Definition \ref{definizione gauge SDE}}, we can restate \textit{Proposition \ref{Invariance under random rotations of the infinitesimal generator}} saying that $ SDE_{\mu',\sigma'} \in \mathcal{Gauge}(SDE_{\mu,\sigma})$ if and only if $ L'=L.$
\end{osservazione}

\section{Finite stochastic transformations of SDEs}\label{Finite_stochastic_transformations_SDE}
A stochastic transformation \( T \) can be regarded as a function acting both on a given SDE and on its (possibly weak) solution \( (X, W) \). Following the notation introduced in \cite{paper2015}, \cite{paper2020}, and \cite{paper2023}, we denote by \( E_T \) the action of \( T \) on the SDE, and by \( P_T \) its action on the solution process. This is done in such a way that the following structure is preserved: if \( (X, W) \) is a (weak) solution to \( SDE_{\mu, \sigma} \), then the transformed process \( P_T(X, W) := (P_T(X), P_T(W)) \) is still a (weak) solution to the transformed equation \( E_T(SDE_{\mu, \sigma}) := SDE_{E_T(\mu), E_T(\sigma)} \). Analogously to the classical distinction between strong and weak solutions to SDEs, one can define strong and weak stochastic transformations. A strong stochastic transformation leaves the Brownian motion unchanged, that is, \( P_T(W) = W \), since the Brownian motion is assumed to be given \emph{a priori} with the filtered probability space. On the other hand, a weak stochastic transformation also modifies the Brownian component, i.e., \( P_T(W) \neq W \), and in this case it is necessary to verify that \( P_T(W) \) still defines a Brownian motion (possibly with respect to a transformed filtered probability space) in order for the transformed SDE to remain well-posed.\\
To define a class of \emph{admissible} stochastic transformations, as done in \cite{paper2015}, \cite{paper2020}, and \cite{paper2023}, we consider transformations that act on the fundamental degrees of freedom of the stochastic setting: time-dependent spatial diffeomorphisms, time changes (both common in deterministic frameworks), gauge symmetries of the Brownian motion, and changes of the underlying probability measure. All these types of transformations have been extensively studied in \cite{paper2015}. However, in \cite{paper2023}, the main result was proved without including random rotations of the Brownian motion among the admissible transformations.\\
For the reader’s convenience, this section briefly recalls the main definitions and results concerning stochastic transformations from \cite{paper}, though using a slightly different approach (see Section~\ref{Invarian under random rotations}). In the following sections, we will generalize the main result of \cite{paper2023} to include random rotations as well.

\begin{definizione}\label{def trasf}
    Given two open subsets $\tilde{M},\tilde{M}'$ of $\mathbb{R}^n$, we denote by  $M=\tilde{M}\times \mathbb{R}_+$ and $M'=\tilde{M}'\times \mathbb{R}_+$. Let $\tilde{\Phi}: \tilde M\rightarrow \tilde{M}'$,  $B: M \rightarrow SO(m)$, $\eta:\mathbb{R}_+\rightarrow \mathbb{R}_+$, and $ h: M  \rightarrow{R^m}$ be smooth functions. If $ f:\mathbb{R}_+ \rightarrow \mathbb{R}_+$  is the functions $ f(t)=\int_0^t \eta(s)ds$, let $\Phi: M\rightarrow M'$ defined by $ \Phi(x,t)=(\tilde{\Phi}(x),f(t))$ be a diffeomorphism.\\
    
    A (finite) stochastic transformation from $ M$ to $ M'$ is a quadruple $ T=(\Phi, B, \eta, h)$, where the specific action of every single component of $ T$ can be summarized as follows.
    \begin{itemize}
        \item The real valued and smooth function $\eta$ describes a (deterministic) time change transformation $T_1$  given by an absolutely continuous function $\ f:\mathbb{R}_+ \rightarrow\mathbb{R}_+$ 
        \begin{equation}
          t'=f(t)=\int_0^t \eta(s)ds\label{deterministic time change}  
        \end{equation} 
        with $\eta $ strictly positive so that $ f(0)=0$. The absolute continuity assumption ensures the validity of the fundamental theorem of calculus and, consequently, the well-posedness of the representation formula (\ref{deterministic time change}), with $\eta:=f'$. 
        The action $ P_{T_1}$  
        on the solution process is the following \[ P_{T_1}(X)=\mathcal{H}_{\eta}(X), \ \ \ \ \ \  P_{T_1}(W)=\mathcal{H}_{\eta}(W'),\ \ \  with \ \ dW'_t=\sqrt{\eta(t)}dW_t\]
        where $\mathcal{H}_{\eta}$ is the functional such that when applied to a generic process $ Y$ gives $ [\mathcal{H}_{\eta}(Y)]_t=Y_{f^{-1}(t)}$. Moreover, the action $ E_{T_1}$ of $ T_1$ on   $SDE_{\mu, \sigma}$ can be expressed as
        \[ E_{T_1}(\mu)=\frac{\mu}{\eta}, \ \ \ E_{T_1}(\sigma)=\frac{\sigma}{\sqrt{\eta}}.\]
        We note that $ T_1$ changes the original filtered probability space switching from filtration $\mathcal{F}_{t}$ to filtration $\mathcal{F}_{f^{-1}(t)}$. Since  well-known results (see \cite{paper2015}) guarantee that $\mathcal{H}_{\eta}(W')$ is a Brownian motion w.r.t. the latter filtration, then $ T_1$ is a well-defined transformation.

        Since in the following we will deal with non-autonomous transformations involving both $ X_t$ and $ t$, 
        it is convenient to extend the functional $\mathcal{H}_{\eta}$ introduced above to the whole pair $ (X_t,t)$, defining a time transformation  that in the following way: \[\mathcal{H}_{\eta}(X_t,t)=( \mathcal{H}_{\eta}(X_t),f^{-1}(t) )=( X_{f^{-1}(t)}, f^{-1}(t) )\]

        \item  The diffeomorphism  $\Phi: M \rightarrow M'$ describe the time dependent spatial transformation $T_2$. The action $ P_{T_2}$ on the solution process is given by \[ P_{T_2}(X)=\Phi(X),\ \ \ \ P_{T_2}(W)=W  \]
        while the action $ E_{T_2}$ on the equation reads as \[ E_{T_2}(\mu)=L(\Phi)\circ \Phi^{-1}, \ \ \ \ \ E_{T_2}(\sigma)=(D(\Phi) \sigma)\circ \Phi^{-1}\]
        where  $ L$ denotes the infinitesimal generator of the SDE.
        The original filtered probability space remains unchanged under the action of $\ T_2.$
        \item The matrix $B$ is responsible of the random rotation transformation of the Brownian motion, which we denote by $T_3$, extensively studied in \textit{Section \ref{Random rotation of Brownian motion}}. The double action of ${T_3}$ is given by
        \[P_{T_3}(X)=X, \ \ \ \ P_{T_3}(W)=\int_0^t B(X_s,s) dW_s\]
        \[ E_{T_3}(\mu)=\mu, \ \ \ \ E_{T_3}(\sigma)=\sigma B^{-1}.\]
        As we have already seen, $ P_{T_3}(W)$ is still a Brownian motion w.r.t. the same probability space which implies that $\ T_3$ is well-defined.
        \item The function $h$ (satisfying $h(X_t,t) \in M^2[0,T]$) provides the stochastic transformation $T_4$ that changes the original filtered probability space,  switching from the underlying probability measure $\mathbb{P}$ to an equivalent probability measure $\mathbb{Q}$ with Radon-Nikodym derivative between the two measures given by  
        \[ \frac{d \mathbb{Q}}{d \mathbb{P}}|_{\mathcal{F}t} = \exp \Big( \int_0^t \sum_{\alpha=1}^m h_{\alpha}(X_s,s) d W_s^{\alpha} - \frac{1}{2} \int_0^t \sum_{\alpha=1}^{m} (h_{\alpha}(X_s,s))^2 ds\Big). \] 
        Non explosivity hypothesis on $ h$ are required in order to ensure the well-posedness of the measure change of \textit{Girsanov}-type. The non explosivity property, in fact, ensures that the solution exists for all times almost surely. If the transformation $T_4$  is such that the transformed process remains non explosive, then one can correctly apply Girsanov theorem,  see \cite{paper2020}( Definition 3.3 and Lemma 3.4) and also \cite{paper2015},\cite{paper2023}. The action $ T_4$ on the process is given by 
        \[ P_{T_4}(X)=X, \ \ \ \ \ P_{T_4}(W)_t=W_t-\int_0^t h(X_t,t)dt,\] 
        while the action $\ E_{T_4}$ on the equation can be represented as
        \[ E_{T_4}(\mu)=\mu+\sigma h, \ \ \ \ \ \ E_{T_4}(\sigma)=\sigma.\]
        The well-known \textit{Girsanov theorem} (see, e.g. \cite{baldi}) ensures that $ P_T(W)$ is a Brownian motion w.r.t. the new probability measure $ \mathbb{Q}$, so $ T_4$ is well-defined.
    \end{itemize}
        We denote by $\ S_m(M,M')$ the set of all stochastic transformations $\ T=(\Phi,B,\eta,h)$ from $\ M$ to $\ M'$ and by $\ S_m(M)$ when $ M^\prime=M.$   
\end{definizione}

The following theorem allows us to unify the four transformations introduced in the previous definition.
\begin{teorema}[Double action of the transformation]\label{trasformate}
    Let $ T=(\Phi, B, \eta, h)$ be a stochastic transformation. Consider  the solution $ X$ to $ SDE_{\mu,\sigma}$ driven by $ W$, with $ SDE_{\mu,\sigma}$ non-explosive and with respect to which $ h$ and $\eta$ are non-explosive, with $ X$ taking values in $ \tilde{M}$ and $ W$ being a Brownian motion in the probability space $ (\Omega, \mathcal{F}, \mathbb{P})$. Then, the transformed process $ P_T(X,W)=(P_T(X),P_T(W))$ is a solution to the transformed equation $E_T(SDE_{\mu,\sigma})=SDE_{E_T{(\mu)},E_T{(\sigma)}}$, with $ P_{T}(X)$ taking values in $ \tilde{M}'$ and $ P_{T}(W)$ being Brownian motion in the space $ ( \Omega, \mathcal{F'}, \mathbb{Q})$ and with
    \begin{equation*}
        P_{T}(X)= \Phi(\mathcal{H}_{\eta}(X,t)); \ \ \ \ \ \  P_{T}(W)= \mathcal{H}_{\eta}(W') 
    \end{equation*}
    \begin{equation*}
        dW'_t = \sqrt{\eta(t)}B(X_t,t)(dW_t - h(X_t,t) dt)
    \end{equation*}
    \begin{equation*} 
        \frac{d \mathbb{Q}}{d \mathbb{P}}_{| \mathcal{F}_t} = \exp \Big( \int_0^T \sum_{\alpha=1}^m h_\alpha(X_s,s) d W^\alpha_s - \frac{1}{2} \int_0^T \sum_{\alpha=1}^m (h_\alpha(X_s,s))^2 ds \Big)
    \end{equation*}
    \begin{equation*}
        \mathcal{F'}_t=\mathcal{F}_{f^{-1}(t)}.
    \end{equation*}
 \begin{equation*}
        E_{T}(\mu)= \Big( \frac{1}{\eta} [ L(\Phi) + D(\Phi) \sigma  h] \Big) \circ \Phi^{-1}; \ \ \ \ \ \ \
        E_{T}(\sigma)= \Big(\frac{1}{\sqrt{\eta}} D(\Phi)  \sigma B^{-1} \Big) \circ \Phi^{-1}.
    \end{equation*}
    
\end{teorema}

\begin{proof}
    The theorem can be easily proved by iterating the actions of the four specific components summarized in the previous definition. See \cite{paper2020} for all the details.
\end{proof}

\begin{osservazione}
 The usual distinction between weak and strong solutions to a given  SDE suggest the introduction of an analogous distinction between weak and strong stochastic transformations. Indeed, an It\^o diffusion $X$ is a weak solution to  $ SDE_{\mu,\sigma}$ if there exists a filtered probability space $ (\Omega, \mathcal{F},\mathcal{F}_t,\mathbb{P})$ and a Brownian motion $ W$ w.r.t this probability space such that $dX_t=\mu(X_t,t)dt+\sigma(X_t,t)dW_t$. Conversely, an It\^o diffusion $X$ is a strong solution to$\ SDE_{\mu,\sigma}$ if $X$ satisfies the previous equation in a  given a priori filtered probability space $ (\Omega, \mathcal{F},\mathcal{F}_t,\mathbb{P})$ with a fixed  given Brownian motion $W$, w.r.t this probability space. Similarly, a strong stochastic transformation acts only on a fixed probability space and w.r.t. a fixed Brownian motion, and cannot change neither of them: since $ h$ changes the underlying probability, $\eta$ changes the filtration and $B$ changes the Brownian motion, to have a strong stochastic transformation we need $h=0$, $\eta=1$ and $ B=I_m$.   
\end{osservazione}

\subsection{Algebraic and geometric structure of stochastic transformations}

In this section we discuss the algebraic structure of the set of finite stochastic transformations and we exploit the corresponding geometrical setting in order to introduce the notion of infinitesimal transformation.\\
Consider the group given by the direct product $( SO(m)  \times \mathbb{R}_+, \times) $ with the following law \[ (B_1, \eta_1) \times (B_2, \eta_2)=(B_2 \cdot_{SO(m)} B_1, \eta_2 \cdot_{\mathbb{R}_+} \eta_1), \]
where $\cdot_{SO(m)}$ and  $ \cdot_{\mathbb{R}_+}$ are the standard products in the matrix space and in $\mathbb{R}_+$, respectively.  
Consider now the semidirect product $ G:=((SO(m)\times \mathbb{R}_+),\times) \ltimes_{\psi} (\mathbb{R}^m,+) $
with the product law $ *$ induced by the homomorphism 
\[ \psi:  (SO(m)\times\mathbb{R}_+ ) \rightarrow Aut(\mathbb{R}^m)\]
\[ (B,\eta) \mapsto \psi_{(B,\eta)} \]
where 
\[ \psi_{(B,\eta)}: h \mapsto \psi_{(B,\eta)}(h):=\sqrt{\eta}B^{-1}h. \]
The law product $*$ in $ ((SO(m)\times \mathbb{R}_+) \ltimes_{\psi} \mathbb{R}^m , *)$ is induced by the homomorphism $\psi$ in the standard way:
\[ ((B_1,\eta_1),h_1)*((B_2,\eta_2),h_2)=((B_1,\eta_1)\times(B_2,\eta_2),h_1+_{\mathbb{R}^m}\psi_{(B_1,\eta_1)}(h_2))\]
i.e.
\[ (B_1, \eta_1, h_1)*(B_2, \eta_2, h_2)=(B_2\cdot B_1, \eta_2\cdot\eta_1,\sqrt{\eta_1}B_1^{-1}h_2 + h_1)\]

\begin{osservazione} The group $ G$ admits a matrix representation, as every element $ g=(B, \eta, h)$ can be identified with the matrix
\begin{equation*} \begin{pmatrix} \sqrt{\eta} B^{-1} & h \\ 0 & 1 \end{pmatrix} . \end{equation*}

Passing to homogeneous coordinates, with this matrix representation $ G$ can be seen as a subgroup of roto-translations  (matrix rotation $ B^{-1}$ and translating factor $ h$) with a positive scaling factor ($\sqrt{\eta})$ on $\mathbb{R}^m$, with composition law given by $ *$. 

\end{osservazione}
\noindent
Since stochastic transformations also include spatial diffeomorphisms, we can consider the manifold $ M \times G$, where $M:=\tilde{M} \times \mathbb{R}_+$, with the trivial principal bundle structure given by the projection $\pi_M$ onto the first component 
\begin{equation*} \pi_M : M \times G  \longrightarrow M \end{equation*} \begin{equation*} ((x,t),g) \mapsto (x,t) \end{equation*} and a right action of $ G$ on $ M \times G$ that leaves $ M $ invariant:
\begin{equation*} R_{M,g}:  M \times G \longrightarrow M \times G \end{equation*} \begin{equation*} ((x,t),g_1) \mapsto ((x,t),g_1*g) \end{equation*}
where $*$ is the product in $G$ defined above.

\begin{definizione} Given two (trivial) principal bundles $ M \times G$ and $  M' \times  G $, an isomorphism between principal bundles is a diffeomorphism $ F:  M \times G \rightarrow M' \times G$ that preserves the principal bundle structure of $ M \times G$ and $ M' \times G$. In other words, there exists a diffeomorphism $\Phi: M \rightarrow M'$ such that 
\begin{equation}\label{isomorfismo fibrati} 
 \pi_{M'} \circ F = \Phi \circ \pi_{M}, \ \ \ \ \ \  \ \ \ \ \ \  F \circ R_{M,g} = R_{M',g} \circ F \ \ \ \forall g\in G.
\end{equation} 
\end{definizione}
 
\begin{osservazione}\label{identificazione} The isomorphism $ F$ is completely determined by the value of $ ((x,t),e)$, where $ e=(I_m,1,\underline{0})$ is the neutral element of the group $ G$. Indeed, $\forall ((x,t),g_1) \in M \times G$, we have that
\begin{equation*} ((x,t),g_1) = ((x,t), e \cdot g_1) = R_{M,g_1}((x,t),e)
\end{equation*} 
thus
\begin{equation*} 
F((x,t),g_1)=F \circ R_{M,g_1}((x,t),e) \underbrace{=}_{(\ref{isomorfismo fibrati})} R_{M',g_1}F((x,t),e). \end{equation*} 
Hence, given the image of $ ((x,t),e)$, to reconstruct the image of a generic element $ ((x,t),g_1)$ it is sufficient to multiply the second component of $ F((x,t),e)$ by $ g_1$. Since by definition there exists $\Phi$ such that $F\circ\pi_{M'} = \pi_{M}\circ\Phi$, the first component of every couple $ ((x,t),g_1)$ is mapped to $\Phi(x,t)$. So,
assuming that $ F((x,t),e)= (\Phi(x,t),g)$, we have that 
\begin{equation*} F((x,t),g_1) = R_{M',g_1}F((x,t),e)= R_{M',g_1}( \Phi(x,t),g) = (\Phi(x,t),g*g_1). 
\end{equation*} 
\end{osservazione}

\noindent
From Remark \ref{identificazione}, it follows that there is a one-to-one correspondence between $ F$ and the pair $ (\Phi(x,t),g)= F( (x,t),e)$, with $ g=(B,\eta,h)$, i.e., there is a natural identification between a stochastic transformation $ T=(\Phi, B, \eta, h)$ and the isomorphism $ F_T: M \times G\rightarrow M' \times G$ defined as 
\begin{equation*} F_T((x,t),g) =\Big(\Phi(x,t), \big(B(x,t),\eta(t), h(x,t)\big)\cdot g  \Big). \end{equation*}

 Thus, $S_m(M,M')$ inherits the properties of the set $Iso(G \times M, G \times M')$: in particular, the natural composition between an element of $ Iso(G \times M, G \times M')$ and an element of $ Iso(G \times M', G \times M'')$ to obtain an element of $ Iso(G \times M, G \times M'')$ guarantees the existence of a natural composition law between elements of $ S_m(M,M') $ and $ S_m(M',M'')$ to obtain an element of $ S_m(M,M'')$. Indeed, we can exploit the identification of stochastic transformations with the isomorphisms of suitable trivial principal bundles in order to define the following algebraic structure on the group of stochastic transformations:
 
 \begin{definizione}[Composition law] \label{legge di composizione} Let us consider two stochastic transformations $T_1=(\Phi_1, B_1, \eta_1, h_1) \in S_m(M,M')$ and $ T_2=(\Phi_2, B_2, \eta_2, h_2)\in S_m(M',M'')$. We define $ T_2\circ T_1$ as the unique transformation corresponding to the composition $F_{T_2}\circ F_{T_1}$ of the associated isomorphisms, i.e. \begin{equation*}
 T_2 \circ T_1 := \big( \Phi_2 \circ \Phi_1, \ (B_2 \circ \Phi_1)  B_1, \ (\eta_2 \circ f_1)  \eta_1,\ \sqrt{\eta_1}B_1^{-1}( h_2 \circ \Phi_1) + h_1 \big) \in S_m(M,M''),
 \end{equation*}
 where $\Phi_2 \circ \Phi_1=(\tilde{\Phi}_2\circ \tilde{\Phi}_1,f_2\circ f_1).$
 In the same way, we define $\ T_1^{-1}$ as the unique transformation corresponding to the inverse isomorphism $\ F_{T_1}^{-1}$, i.e.
 \begin{equation*}
        T_1^{-1}:=( \Phi_1^{-1}, \ (B_1 \circ \Phi_1^{-1})^{-1}, \ (\eta_1 \circ f_1^{-1})^{-1}, - \frac{1}{\sqrt{\eta_1}} B_1  h_1 \circ \Phi_1^{-1}) \in S_m(M',M).
    \end{equation*}
 \end{definizione}

The geometric description of stochastic transformations has a deep probabilistic foundation, which ensures a correspondence between the composition of transformations just described and the composition in terms of transformed processes and the transformation of the coefficients of an SDE. In particular we have the following result.

\begin{teorema}\label{composizione di processi}
    Let $ T\in S_m(M,M')$ and $ T'\in S_m(M',M'')$ be two stochastic transformations and $ SDE_{\mu,\sigma}$ be a non-explosive stochastic differential equation such that $ E_{T}(\mu,\sigma)$ and $ E_{T'}(E_{T}(\mu,\sigma))$ are non-explosive.  Given  a weak solution $(X,W)$  to $ SDE_{\mu,\sigma}$ in the probability space $ (\Omega, \mathcal{F}, \mathbb{P})$, then in the probability space $\ (\Omega, \mathcal{F}, \mathbb{Q})$
    it holds that 
    \begin{equation*} P_{T'}\circ P_{T}(X,W) = P_{T' \circ T}(X,W); \ \ \ \ \ \  E_{T'} \circ E_T (\mu,\sigma)= E_{T' \circ T}(\mu,\sigma) \end{equation*}
\end{teorema}
\begin{proof}
    See \cite{paper2020} for the proof.
\end{proof}

It is easy to verify that the set of stochastic transformations on the same manifold $ S_m(M)$ forms a group with respect to the composition '$\circ$'. Moreover, the identification between the group of stochastic transformations and the group of isomorphisms between bundles $Iso(G\times M, G \times M)$, which is a closed subgroup of the group of diffeomorphisms of $G \times M$, suggests considering the corresponding Lie algebra $ V_m(M)$. Indeed, with the identifications just made for the composition law and for the inverse, it is  possible to introduce the one-parameter group of transformations $ (F_T)_{T \in S_m(M)}$:  \begin{equation*}
F_{T_1}\circ F_{T_2} = F_{T_1 \circ T_2}; \ \ \ \ F_{T}^{-1} = F_{T^{-1}}; \ \ \ \ F_{e}(x)=x, 
\end{equation*} 
where $e$ denotes the identity transformation.   
Moreover, it is possible to introduce the one-parameter group of stochastic transformations $ T_{\lambda}=(\Phi_{\lambda}, B_{\lambda}, \eta_{\lambda}, h_{\lambda})$ indexed in $ (\mathbb{R}, +)$, where $\Phi_{\lambda}=(\tilde{\Phi}_{\lambda}, f_{\lambda})$, with the following notations derived from those of the isomorphisms of bundles:
\begin{equation*} 
\begin{matrix} T_{\lambda_1 + \lambda_2}= T_{\lambda_1}\circ T_{\lambda_2}; & T_{-\lambda}=T_{\lambda}^{-1;} & T_0=e. \end{matrix} 
\end{equation*} 
Thus, $(T_{\lambda})_{\lambda \in \mathbb{R}}$ satisfies all the conditions to be a one-parameter group of transformations  and we can consider the corresponding Lie algebra $ V_m(M)$. We recall that, given a Lie group $G$, its Lie algebra is isomorphic to the tangent space at the identity of $ G$: $\mathcal{g} \cong TG_{|_e}$. Therefore, the elements of the Lie algebra can be identified with  the tangent vectors at the identity ($\lambda=0$) to the maximal integral curves that constitute the Lie group, that is,  the derivatives with respect to the parameter of the one-parameter group elements evaluated at $ 0$. 
Thus,  $V_m(M)$ is constituted by the elements $ V= (Y, C, \tau, H) $, with $Y=(\tilde{Y},m)$ a vector field on $M$, $C: M \rightarrow \mathcal{so}(m) $, $\tau: \mathbb{R}_+ \rightarrow \mathbb{R}$, $ H: M \rightarrow \mathbb{R}^m$ smooth functions obtained in the usual way

\begin{equation}\label{algebra di lie} \begin{matrix} 
Y(x,t)= \partial_{\lambda}(\Phi_{\lambda}(x,t))|_{\lambda=0} & m(t)= \partial_{\lambda}(f_{\lambda}(t))|_{\lambda=0}\\
C(x,t)= \partial_{\lambda}(B_{\lambda}(x,t))|_{\lambda=0}&
\tau(t)=\partial_{\lambda}(\eta_{\lambda}(t))|_{\lambda=0}&  
H(x,t)=\partial_{\lambda}(h_{\lambda}(x,t))|_{\lambda=0} \\  
\end{matrix} \end{equation}

\begin{osservazione}
    Using (\ref{algebra di lie}) and the smoothness of  $\eta$ \footnote{More precisely, the interchange of the derivative and the expectation is formally  justified by Kunita lemma, stated in (\textit{Lemma \ref{Kunita}}).}, we have that
    \begin{equation*}
       m(t)= [\partial_{\lambda} f_{\lambda} (t)]|_{\lambda=0}= \partial_{\lambda}\int_0^t\eta_{\lambda}(s)ds|_{\lambda=0}=   \int_0^t \partial_{\lambda}  \eta_{\lambda}(s)|_{\lambda=0} ds= \int_0^t \tau(s) ds.
  \end{equation*}
\end{osservazione}
\noindent Therefore we can give the following 
\begin{definizione}
An infinitesimal  stochastic transformation is an element $V=(Y, C, \tau,H) $ of the Lie algebra $ V_m(M)$, obtained as in  (\ref{algebra di lie}), with $Y=(\tilde{Y},m)$, where $ \tilde{Y}$ is a vector field on $ \tilde{M}$, $ m: \mathbb{R}_+\rightarrow \mathbb{R}$, $ C: M \rightarrow \mathcal{so}(m)$, $\tau:\mathbb{R}_+\rightarrow \mathbb{R}$, $\ H: M\rightarrow \mathbb{R}^m$. If $ V$ is of the form $ V=((\tilde{Y},0),0,0, 0)$, then it is called a strong infinitesimal stochastic transformation. \end{definizione} 

\noindent
We have seen so far how to compute the Lie algebra $ V_m(M)$ of the given one-parameter group $ T_{\lambda}.$ On the other hand,  given an element of the Lie algebra, one can trace back to the one-parameter group that generated it.

\begin{teorema}[Reconstruction of the flow] \label{teo ricostruzione del flusso} Let $ V=(Y, C, \tau, H)$ be an infinitesimal stochastic transformation (as before, $Y=(\tilde{Y},m)$). Then, it is possible to reconstruct the corresponding one-parameter group $ T_{\lambda}$ by choosing $\Phi_{\lambda}, f_{\lambda}, B_{\lambda},\eta_{\lambda},h_{\lambda}$ as solutions to the following system
 \begin{equation}\label{ricostruzione del flusso} \begin{matrix} \partial_{\lambda}(\Phi_{\lambda}(x,t))=Y(\Phi_{\lambda}(x,t)) & \partial_{\lambda} f_{\lambda}(t)=m(f_{\lambda}(t))\\
 \partial_{\lambda}(B_{\lambda}(x,t))=C(\Phi_{\lambda}(x,t)) B_{\lambda}(x,t) & \partial_{\lambda}(\eta_{\lambda}(t))=\tau(f_{\lambda}(t))\eta_{\lambda}(t)\\ \partial_{\lambda}(h_{\lambda}(x,t))=\sqrt{\eta_{\lambda}}B_{\lambda}^{-1}(x,t) H(\Phi_{\lambda}(x,t)) & \\\end{matrix} \end{equation} with initial conditions \begin{equation*} \begin{matrix} \Phi_0=Id_{\mathbb{R}^n};& f_0=id_{\mathbb{R}}; & B_0=I_m; & \eta_0=1; & h_0=0. \end{matrix} \end{equation*} \end{teorema}

 \begin{proof}
 The proof can be found in  \cite{paper2020}.    
 \end{proof}

\section{ Invariance properties of an SDE}\label{InvarianceProperties}
The concept of transformation is closely linked to the concept of symmetry, where a symmetry is understood as a transformation under which the transformed object remains invariant. 
So far, in \cite{paper}, \cite{paper2020}, and \cite{paper2023}, the following notion of invariance for an SDE has been assumed: an SDE is said to be (strongly/weakly) invariant under the action of a transformation $ T$ if $ T$ preserves the set of its (strong/weak) solutions, that is, if the transformed solution process remains a solution to the original SDE. As we will see later, if we allow gauge transformations such as random rotations of the driving Brownian motion, a novel notion of invariance arises.

\begin{definizione}\label{PT}
    A stochastic transformation $ T \in S_m(M)$ is a weak (finite) symmetry of  $SDE_{\mu,\sigma}$ if, for every solution process $ (X, W)$, the transformed process $$ P_T(X,W)=(P_T(X),P_T(W))$$ 
    is still a solution to $SDE_{\mu,\sigma}$. \\
    A strong stochastic transformation $ T \in S_m(M)$ is a strong (finite) symmetry of $ SDE_{\mu,\sigma}$ if for every solution $ (X,W)$ to $ SDE_{\mu,\sigma}$ the transformed process $ P_T(X,W)=(P_T(X),W)$  is still a solution to $ SDE_{\mu,\sigma}$.
\end{definizione}

\noindent The following theorem, proved in \cite{paper2020}, provides a characterization of the symmetry property in terms of the coefficients of an SDE:
\begin{teorema}\label{equazioni determinanti finite}
    A stochastic transformation $\ T=(\Phi, B, \eta,h) \in S_m(M)$ is a (weak) symmetry of  $ SDE_{\mu,\sigma}$ if and only if
   
\begin{equation}\label{simmetria estesa 1}
    \mu=\Big(\frac{1}{\eta} [ L(\Phi) + D(\Phi)  \sigma h]\Big)\circ\Phi^{-1}, \ \ \ \ \ \  \sigma=\Big(\frac{1}{\sqrt{\eta}} D(\Phi)  \sigma B^{-1} \Big)\circ\Phi^{-1}.
\end{equation}

\end{teorema}

For later use, we recall some results about infinitesimal symmetry.

\begin{definizione}
    A stochastic infinitesimal transformation $ V$ that generates a one-parameter group $ T_{\lambda}$ is called an infinitesimal symmetry - respectively, strong or weak - for the $SDE_{\mu,\sigma}$ if $T_{\lambda}$ is a finite symmetry - respectively, strong or weak - of the $ SDE_{\mu,\sigma}$.
\end{definizione}

The following theorem provides a useful characterization for the infinitesimal symmetries of an SDE, presenting \textit{determining equations}. These equations are the infinitesimal counterpart of the equations \eqref{simmetria estesa 1}, but they provide a powerful tool for the explicit calculation of the symmetries of an SDE, being much simpler to solve than \eqref{simmetria estesa 1}.

\begin{teorema}\label{equazioni determinanti}
Let $ V=(Y,C,\tau,H)$ be a stochastic infinitesimal transformation. Then $ V$ is an infinitesimal symmetry of $ SDE_{\mu,\sigma}$ if and only if $ V$ generates a one-parameter group on $ M$ and the following conditions hold:
\begin{equation}\label{equazione determinante 1.1}
    Y(\mu) - L(Y)  - \sigma  H + \tau \mu=0, \ \ \ \ \  [Y, \sigma] + \frac{1}{2} \tau \sigma + \sigma C=0.
\end{equation}
where $\ [Y, \sigma]^i_{\alpha}=Y(\sigma^i_{\alpha}) - \partial_k(Y^i)\sigma^k_{\alpha}= Y^k \partial_k (\sigma^i_{\alpha}) - \partial_k(Y^i)\sigma^k_{\alpha}. $

\end{teorema}
\begin{proof}
    The main idea is deriving (\ref{simmetria estesa 1}) w.r.t. the parameter. See ~\cite{paper2020} for a complete proof.
\end{proof}
\subsection{A novel notion of invariance of an SDE}

According to Definition \ref{PT}, we defined a symmetry as a transformation which leaves the SDE invariant, where, as 'invariant', we meant 'exactly the same SDE'. This notion, that is very close to the one applied in the deterministic framework, led to determining equations \eqref{simmetria estesa 1}, characterizing a symmetry $ T$, basically preserving the SDE coefficients $\mu$ and $\sigma:$ \[E_T(\mu)=\mu \ \ \ E_T(\sigma)=\sigma.\]
On the other hand, this is not the only possible choice for a notion of invariance for an SDE. Indeed, since Brownian motion is rotational-invariant, given a smooth function $ B: M \rightarrow SO(m)$ and  a weak solution $ (X,W)$ to $ SDE_{\mu,\sigma}$, then $(X,W')$ is a weak solution to $\ SDE_{\mu, \sigma B^{-1}}$, where $ W'_t=\int_0^t B(X_s,s) dW_s$. Recalling Proposition \ref{Invariance under random rotations of the infinitesimal generator}, since two SDEs differ for a rotation of Brownian motion if and only if they share the same infinitesimal generator, we have that weak solutions to a Brownian motion driven SDE are not identified by the coefficients $\mu,\sigma$ of the SDE but by the generator $ L$ given in \eqref{L}. Hence, we can consider a new notion of invariance for an SDE, and a corresponding new notion of symmetry. 

\begin{definizione}\label{G 1}
A (weak) stochastic transformation $ T=(\Phi, B, \eta,h)$ is a $\mathcal{G}-$weak symmetry of a given $ SDE_{\mu,\sigma}$ if, for any $ X$ solution of the martingale problem $(\mu,\sigma\sigma^T)$,  $ P_T(X)$ remains a solution of the same martingale problem. 
\end{definizione}
\begin{osservazione}
    The notion of $\mathcal{G}-$weak symmetry embraces the one of weak symmetry given in the previous section. The converse is not true: a $\mathcal{G}-$weak symmetry is not in general a weak symmetry. 
\end{osservazione}
Once again, it is possible to characterize this new notion of symmetry in terms of the coefficients of an SDE:
\begin{teorema}\label{finite_transformation_coefficient}
    A stochastic symmetry $\ T=(\Phi, B,\eta,h) \in S_m(M)$ is a $\mathcal{G}-$weak symmetry of $\ SDE_{\mu,\sigma}$ if and only if
    \begin{equation}\label{g weak fin}
        \mu=\Big(\frac{1}{\eta}[L(\Phi)+D(\Phi)\sigma h]\Big)\circ\Phi^{-1}, \ \ \ \ \ \ \ \ \ \ \sigma \sigma^T=\Big(\frac{1}{\eta} D(\Phi) \sigma \sigma^T D(\Phi)^T\Big) \circ \Phi^{-1}.
    \end{equation}
\end{teorema}
\begin{proof}
    The proof in an immediate application of Theorem 8.4.3 in \cite{oksendal}.
\end{proof}
Even in this case, determining equations are more effective in the infinitesimal setting, since they are more simple to solve than (\ref{g weak fin}).
\begin{teorema}\label{new-determining-equation}
    Let $\ V=(Y,C,\tau, H)$ be a stochastic infinitesimal transformation. Then $\ V$ is an infinitesimal $\mathcal{G}-$ weak symmetry of $\ SDE_{\mu,\sigma}$ if and only if $\ V$ generates a one-parameter group on $\ M$ and the following determining equations hold:
    \begin{equation}\label{g weak inf}
    Y(\mu)-L(Y)-\sigma H + \tau \mu =0, \ \ \ \ [Y,\sigma \sigma^T]+\tau \sigma \sigma^T=0,      
    \end{equation}
    where $\ [Y,\sigma \sigma^T]=Y(\sigma \sigma^T)-D(Y)\sigma \sigma^T -\sigma \sigma^T D(Y)^T$.
\end{teorema}
\begin{proof}
    If $\ V$ is an infinitesimal $\mathcal{G}-$weak symmetry, by costruction it generates a one-parameter group $\ T_{\lambda}=(\Phi_{\lambda}, B_{\lambda}, \eta_{\lambda}, h_{\lambda})$ which satisfies equations (\ref{g weak fin}): \[  \mu=\Big(\frac{1}{\eta_{\lambda}}[L(\Phi_{\lambda})+D(\Phi_{\lambda})\sigma h_{\lambda}]\Big)\circ\Phi_{\lambda}^{-1}, \ \ \ \ \ \ \ \ \ \ \sigma \sigma^T=\Big(\frac{1}{\eta_{\lambda}} D(\Phi_{\lambda}) \sigma \sigma^T D(\Phi_{\lambda})^T\Big)\circ\Phi_{\lambda}^{-1}.\]
    After deriving these equations with respect to the parameter $\lambda$ and applying straightforward Theorem \ref{teo ricostruzione del flusso}, we obtain (\ref{g weak inf}).\\
    Conversely, suppose that (\ref{g weak inf}) holds and that $ V$ generates a one-parameter group $ T_{\lambda}=(\Phi_{\lambda}, B_{\lambda}, \eta_{\lambda}, h_{\lambda}).$ The first equation is the same as in Theorem \ref{equazioni determinanti}, and the proof is identical; therefore, we refer the reader to \cite{paper2015}. Concerning the second equation, let us define
    define
    \[ \sigma_{\lambda}:=\Big(\frac{1}{\sqrt{\eta_{\lambda}}}D(\Phi_{\lambda})\sigma B_{\lambda}^{-1}\Big)\circ\Phi_{\lambda}^{-1}\ \ \ \ \ \ \ \ \ \ S_{\lambda}:=\Big(\frac{1}{\eta_{\lambda}} D(\Phi_{\lambda}) \sigma \sigma^T D(\Phi_{\lambda})^T\Big)\circ\Phi_{\lambda}^{-1}\]
    and consider $\tilde{\sigma}_{\lambda}:=\sigma_{\lambda}\circ \Phi_{\lambda}, \tilde{S}_{\lambda}:=S_{\lambda}\circ \Phi_{\lambda}.$ Then from Theorem \ref{teo ricostruzione del flusso} we have
    \[\partial_\lambda \tilde{S}_\lambda =
- (\tau\circ \Phi_{\lambda} )\tilde{S}_\lambda + (D(Y) \circ \Phi_\lambda) \cdot \tilde{S}_\lambda + \tilde{S}_\lambda \cdot (D(Y)^\top \circ \Phi_\lambda),\]
from which
\[  \partial_\lambda S_\lambda = -[Y, S_\lambda] - \tau S_\lambda,\]
 that is, for $ x$ fixed, an ordinary differential equation in $\lambda$ admitting a unique solution for any initial condition $ S_0=\sigma\sigma^T.$ Furthermore, since $\partial_\lambda S_0=0$ and since (\ref{g weak inf}) holds, then $ S_0$ solves the ODE above, and so, by uniqueness of solutions, $ S_\lambda=S_0$, that is, $ S_{\lambda}=\sigma\sigma^T$, i.e.
 \[\sigma \sigma^T=\Big(\frac{1}{\eta_{\lambda}} D(\Phi_{\lambda}) \sigma \sigma^T D(\Phi_{\lambda})^T\Big)\circ\Phi_{\lambda}^{-1}. \]But this implies that (\ref{g weak fin}) holds: $V$ generates a $\mathcal{G}-$weak finite symmetry $ T_{\lambda}$, and so $V$ is a $\mathcal{G}-$weak infinitesimal symmetry.
\end{proof}
\begin{osservazione}
    We already noticed that the notion of $\mathcal{G}-$weak symmetry is more general than the one of weak symmetry. Indeed, \eqref{equazione determinante 1.1} implies \eqref{g weak inf} since $C$ is an antisymmetric matrix. The equivalence between these two notions of symmetry can be proved only under additional regularity conditions on $\sigma^T \sigma$, namely its invertibility. In fact, using \eqref{g weak fin} we can write $\sigma$ as \[\sigma=\sigma\sigma^T \sigma (\sigma^T\sigma)^{-1}=\Big(\frac{1}{\eta}D(\Phi)\sigma\sigma^TD(\Phi)^T)\circ \Phi^{-1}\sigma (\sigma^T \sigma)^{-1},\]
and to ensure that $\ T$ is also a weak symmetry (i.e. to satisfy the condition in Theorem \ref{equazioni determinanti finite}), we are forced to impose \[ B^{-1}\circ \Phi^{-1}=\Big(\frac{1}{\sqrt{\eta}}\sigma^T D(\Phi)^T\Big)\circ \Phi^{-1}\sigma (\sigma^T\sigma)^{-1},\]
    from which $ B$ is fully recovered. An easy calculation shows that under this choice $ B$ is orthogonal, as required. The question of the invertibility of $\sigma^T\sigma (x)$ arises as part of the broader problem referred to in the literature as the \textit{square root of a matrix field} (see, e.g., \cite{Rpgerswilliams}, Chapter 2, Section 12 and the discussion in Section ~\ref{Invariance under random rotations of the infinitesimal generator}).
\end{osservazione}
\begin{osservazione}
Neither the finite nor the infinitesimal determining equations involve the rotation matrix (i.e. $B$ and $C$ in the finite and infinitesimal settings, respectively). As a result, solving these equations yields symmetries without any rotational component. Nonetheless, any rotation is itself a $\mathcal{G}$-weak symmetry, so precomposing a three-component $\mathcal{G}$-weak symmetry with \emph{any} rotation gives rise to a four-component $\mathcal{G}$-weak symmetry.\\
A particularly relevant case is when a transformation $T = (\Phi, B, \eta, h)$ is such  that $T_1 = (\Phi, \eta, h)$ is a weak symmetry in the sense of Definition~\ref{PT}. Then, $T$ can be written as
\[
T = R_{B \circ \Phi^{-1}} \circ T_1,
\]
where $R_{B \circ \Phi^{-1}} = (\mathrm{Id}, B \circ \Phi^{-1}, 1, 0)$ is a rotation (i.e., a gauge symmetry), and $T_1 = (\Phi, I_m, \eta, h)$ is a three-component weak symmetry. Since $T_1$ is a weak symmetry, it follows that $T$ is a $\mathcal{G}$-weak symmetry. In other words, this specific subclass of $\mathcal{G}$-weak symmetries—obtained by composing a gauge symmetry with a three-component weak symmetry—can be fully characterized without involving the matrix $B$ in the determining equations. In this case, the infinitesimal determining equations for $\ T$ reduce to those of standard weak symmetries for $\ T_1$, i.e. to ~\eqref{equazione determinante 1.1}  with the infinitesimal rotation matrix $C$ identically zero.
\end{osservazione}

As it will be explained in the following sections, this more general notion of symmetry is useful to better explain the rotational invariance of the integration by parts formula. Moreover, since the proofs of the quasi-invariance principle and of the integration by parts formula lie on the invariance of the law of the solution process under the action of a symmetry, $\mathcal{G}-weak$ symmetry is the most general notion of symmetry under which the two theorems remain valid.

\section{Quasi-invariance principle and analytical setting}\label{quasi invariance principle}

\subsection{A quasi-invariance principle via Lie symmetry theory}\label{quasi-invariance principle}
An analogue of the Bismut-Malliavin quasi-invariance principle was established in \cite{paper2023} using Lie symmetry theory applied to SDEs. The symmetries considered include spatial diffeomorphisms, deterministic time changes and Girsanov-type measure changes. In this work, we extend the framework by including Brownian motion rotations among the admissible transformations. This leads to a new derivation of the invariance principle and the associated integration by parts formula. Although the formulas remain formally unchanged, their derivation requires addressing new, non-trivial analytical terms.

\begin{teorema}[Quasi-invariance principle]\label{principio di quasi-invarianza}
Let $ (X,W)$ be a weak solution to a non-explosive $ SDE_{\mu,\sigma}$ and let $ (X',W')=(P_T(X),P_T(W))$ be the transformed process via the finite stochastic transformation $ T=(\Phi, B, \eta,h)$, where $ h$ is a predictable stochastic process that forms a non-explosive vector field for $SDE_{\mu,\sigma}$. If $T$ is a (finite) symmetry of $SDE_{\mu,\sigma}$, then, for any fixed time $ t \in [0,T]$ and for any measurable and bounded function $ g$, the following quasi-invariance principle holds: \begin{equation}\label{quasi-invarianza} \mathbb{E}_{\mathbb{P}}[g(X)]=\mathbb{E}_{\mathbb{P}}\Bigg[g(X') \exp\Big(\int_0^t h_{\alpha}(X_s,s) dW^{\alpha}_s - \frac{1}{2}\int_0^t h_{\alpha}^2(X_s,s)ds\Big) \Bigg]
\end{equation}
\end{teorema}
\begin{proof} 
Since $ T$ is a symmetry, the transformed process $ (X',W')$ is a solution to the same $ SDE_{\mu,\sigma}$ (cf. Definition \ref{PT}) and, in particular, $ X'$ under $\mathbb{Q}$ has the same law as $ X$ under $\mathbb{P}$, from which 
\begin{equation}\label{a}
\mathbb{E}_{\mathbb{P}}[g(X)]= \mathbb{E}_{\mathbb{Q}}[g(X')]. 
\end{equation}
Applying the Radon-Nikodym theorem and recalling that
\[
\frac{d\mathbb{Q}}{d\mathbb{P}}|_{\mathcal{F}t}=\exp\Big(\int_0^t h_{\alpha}(X_s,s) dW^{\alpha}_s - \frac{1}{2}\int_0^t h_{\alpha}^2(X_s,s)ds\Big)
\]
for every fixed $ t \in [0,T]$, then (\ref{a}) can be rewritten as \begin{equation*} 
\mathbb{E}_{\mathbb{P}}[g(X)]= \mathbb{E}_{\mathbb{P}}\Bigg[g(X')\exp\Big(\int_0^t h_{\alpha}(X_s,s) dW^{\alpha}_s - \frac{1}{2}\int_0^t h_{\alpha}^2(X_s,s)ds\Big)\Bigg].
\end{equation*}
\end{proof} 
\begin{osservazione} 
The regularity and boundedness assumptions on \( g \), together with the non-explosiveness assumptions, guarantee that both sides of \eqref{quasi-invarianza} are well defined. In particular, the non-explosiveness of \( h \) ensures that \( \frac{d\mathbb{Q}}{d\mathbb{P}} \) is a local martingale, allowing the application of Girsanov’s theorem to establish the equivalence of the probability measures \( \mathbb{P} \) and \( \mathbb{Q} \).\end{osservazione} 
\begin{osservazione}
The proof mirrors that of \textit{Theorem (25) in \cite{paper2023}}: the inclusion of Brownian rotation does not affect the Radon-Nikodym density, and the argument relies solely on the symmetry property of the transformation \( T \), independently of the components it acts upon. As discussed in Remark~\ref{remark generatore infinitesimale}, this invariance reflects the well-known invariance of the SDE’s infinitesimal generator under Brownian motion rotations. This observation motivates the introduction of gauge symmetry, regarded as the most general class of symmetries for which the results established in this paper remain valid.
\end{osservazione} 

\noindent In Malliavin-Bismut calculus, the fundamental invariance principle yields an integration by parts formula; a similar result holds here using Lie symmetries. For simplicity, this formula is proved only for functionals of processes at single times\footnote{The generalization to the case of cylindrical functionals of diffusion processes is discussed in \cite{paper2023}}. Some preliminary analytical regularity results, necessary for stating and proving the formula, will be addressed in the next section.

\subsection{Preliminary stochastic and analytical considerations}\label{analytical considerations}
Although closely related, quasi-invariance has an elegant and elementary formulation, while the integration by parts formula requires regularity assumptions and preliminary analytical results to ensure both its well-posedness and the development of its proof. \\

We start by proving the following result on the time change formula for Itô integrals, which slightly generalizes the result in \cite{oksendal} by also including rotation matrices:
\begin{lemma}[Time change formula for It\^o integrals]\label{Oksendal}
 \textcolor{white}{,}\\   Suppose $\eta(s,\omega)$, $\ B(s,\omega)$ and $\alpha(s,\omega)$ are $ s$-continuous, $\alpha(0,\omega)=0$ for a.a. $\omega$, $\mathbb{E}[\alpha_t]<\infty$ and $ B(s,\omega) \in SO(m)$. Let $ W_s$ be an $ m$-dimensional Brownian motion and let $ v(s,\omega)$ be bounded and $ s$-continuous. Define
 \begin{equation}\label{Btilde}
     \tilde{W}_t=\lim_{j \rightarrow \infty} \sum_{j} \sqrt{\eta(\alpha_j,\omega)} B(\alpha_j,\omega)\Delta W_{\alpha_j} = \int_0^{\alpha_t}  \sqrt{\eta(s,\omega)} B(s,\omega) dW_s.
 \end{equation}
 Then  $\tilde{W}_t$ is an $\mathcal{F}_{\alpha_t}$-Brownian motion and
 \begin{equation}\label{cambio tempo stocastico}
     \int_0^{\alpha_t} v(s,\omega) dW_s = \int_0^t v(\alpha_r,\omega) \sqrt{\alpha'_r(\omega)} B^{-1}(\alpha_r,\omega) d\tilde{W}_r \ \ \mathbb{P}-a.s.,
 \end{equation}
 where $ \alpha'_r(\omega)=\frac{1}{\eta(\alpha_r,\omega)} .$
\end{lemma}
\begin{proof}
    The existence of the limit and the second identity in \eqref{Btilde}, together with the fact that $\tilde{W}$ is still a Brownian motion, have already been proved in Section \ref{Random rotation of Brownian motion}. It remains to prove \eqref{cambio tempo stocastico}: using the fact that It\^o integrals can be seen as limits of Riemannian sums \footnote{c.f.r. Proposition 7.4 in Baldi-Stochastic calculus}, we have 
    \begin{equation*}
        \int_0^{\alpha_t} v(s,\omega) dW_s=\lim_{j \rightarrow \infty} \sum_j v(\alpha_j, \omega) \Delta W_{\alpha_j}=
        \end{equation*}
    \begin{equation*}
        = \lim_{j \rightarrow \infty} \sum_j v(\alpha_j, \omega)\sqrt{\frac{1}{\eta(\alpha_j,\omega)}} B(\alpha_j,\omega)^{-1} \Delta \tilde{W}_j=\int_0^t v(\alpha_r,\omega)\sqrt{\frac{1}{\eta(\alpha_r,\omega)}} B(\alpha_r,\omega)^{-1} d\tilde{W}_r.
    \end{equation*}
\end{proof}
In the following it will be useful to remember the conditions under which it is possible to transport the derivative with respect to a parameter inside an expectation.

\begin{lemma}\label{derivata sotto valore atteso}
  Let $\ g(\lambda, X) : \mathbb{R} \times M \rightarrow \mathbb{R}$ be an integrable function for all $\lambda$, differentiable with continuity with respect to $\lambda$ and such that $\mathbb{E}_{\mathbb{P}}[|\partial^2_{\lambda}g(\lambda,X)|]<\infty.$ Then
  \begin{equation*}
      \partial_{\lambda}\mathbb{E}_{\mathbb{P}}[g(\lambda,X)] = \mathbb{E}_{\mathbb{P}}[\partial_{\lambda}g(\lambda,X)].
  \end{equation*}
\end{lemma}
\begin{proof}
See \cite{paper2023}, \textit{Section 5}.
\end{proof}
\noindent
The following lemma concerns continuity and differentiability of stochastic integrals with respect to parameters (see \cite{kunita} for details). 
\begin{lemma}[Kunita]\label{Kunita}
    Let $\Sigma=\mathbb{R}^e$ be the parameter set and let $\ p > min(e,2)$. Suppose that
    \begin{equation}\label{lemma tecnico 1}
        \sup_{\lambda} \mathbb{E}\Bigg[\int_0^T |g_{\lambda}(r)|^p dr \Bigg] < \infty], \ \ \mathbb{E}\Bigg[\int_0^T |g_{\lambda}(r)-g_{\lambda'}(r)|^p dr \Bigg] < c_p |\lambda - \lambda'|^p
    \end{equation}
    where $\ g_{\lambda}(r)$ is $\ n$-times continuously differentiable with respect to $\lambda$ and, for $\ |i|\leq n$, the derivatives $\partial_{\lambda}^ig_{\lambda}(r)$ satisfy $\ (\ref{lemma tecnico 1})$. Then the family of stochastic integrals $\ (\int_0^t g_{\lambda}(r) dW_r)_t$ has a modification
which is $\ n$-times continuously differentiable with respect to $\lambda$ a.s. Forall $\ |i| \leq n$, we have that
    \begin{equation*}
        \partial^i_{\lambda}\int_0^t g_{\lambda}(r) dW_r = \int_o^t \partial^i_{\lambda}g_{\lambda}(r) dW_r \  \ \forall (t,\lambda) \ a.s.
    \end{equation*}
\end{lemma}
\noindent
The next result pertains to the non-explosivenessess of a solution to an SDE. Since it is based on the theory of Lyapunov functions, it is useful to recall the following. \\
\textbf{Lyapunov condition:} Given  an infinitesimal generator $L$ of the form \eqref{L}, there exists $\phi \in C^{1,2}(\mathbb{R}^n \times [0,T])$, called Lyapunov function, such that $\phi \geq 0$
and
\begin{equation*}
    \lim_{r \rightarrow \infty}(\inf_{0\leq t \leq T, \  |x| \leq r} \phi(x,t)) = \infty \ \ \ \ \ \ \ \  (\partial_t + L) \phi(x,t) \leq M \phi(x,t)
\end{equation*}
where the second condition is given for some constant $ M$, a.s. on $\mathbb{R}^n \times [0,T]$.
\begin{teorema}[Lyapunov]\label{Lyapunov}
    Let $\ a_{ij}:=(\sigma \sigma^T)^{ij} \in W^{1,p}_{loc}(\mathbb{R}^n) \cup C(\mathbb{R}^n)$ 
    and $\mu \in L^p_{loc}(\mathbb{R}^n), p \in (n,\infty).$ A sufficient condition for the martingale problem associated with the infinitesimal generator $ L$ to be well-posed is that, for every $ T > 0$, there exists a number $ M = M_t > 0$ and a non-negative function $\phi \in C^{1,2}(\mathbb{R}^n \times [0,T]) $ such that $\phi$ is a Lyapunov function on $\mathbb{R}^n \times [0,T]$. In this case, the solution process is non-explosive and, for every $ (x,t) \in \mathbb{R}^n \times [0,T]$, the following holds
    \begin{equation*}
        \mathbb{E}_x[\phi(X_t,t)] \leq \exp(Mt) \phi(X_0,0).
    \end{equation*}
\end{teorema}
\begin{proof} See, e.g., \cite{paper2023}.
\end{proof}
The following analytical result is preparatory to the integration by parts theorem, and it focus on verifying in a particular case the assumptions of Lemma \ref{derivata sotto valore atteso}, namely to establish under which conditions it is possible to bring the differentiation inside an expected value. \\

The regularity assumptions assumed in the present paper are the following.
\\
\\
{Hypotesis A'} \textit{Forall} $ X_t$ \textit{solution to} $ SDE_{\mu, \sigma}$, \textit{with initial deterministic conditions, 
let us assume the following regularity hypothesis:}
\begin{equation*}
   C_{\alpha, k} H_{k}(X_t,t),  H_{\alpha}(X_t,t), Y(H_{\alpha}) (X_t,t), L(Y^i)(X_t,t), \Sigma_{\alpha}(Y^i)(X_t,t),
   \end{equation*}
   \begin{equation}\label{ipotesi A'} 
   L(Y(Y^i))(X_t,t), \Sigma_{\alpha}(Y(Y^i))(X_t,t) \in L^2(\Omega).
\end{equation}
\begin{osservazione}
    Notice that, compared to the regularity assumptions used in \cite{paper2023}, Hypothesis A' \eqref{ipotesi A'} adds a regularity condition on the term $ C_{\alpha k}H_k$. If there are no rotations (i.e. if we are in the setting already studied in \cite{paper2023}), then $ B_{\lambda}=I_m$ and $ C$ is the null matrix: the term $ CH $ vanishes and \textit{Hypothesis A'} reduces to the \textit{Hypothesis A} assumed in \cite{paper2023}.
\end{osservazione}

\begin{lemma}\label{lemma tecnico mb}
   Consider an $ SDE_{\mu,\sigma}$ with infinitesimal symmetry $ V=(Y, C, \tau, H)$ satisfying Hypothesis A' (\ref{ipotesi A'}). Let $ T_{\lambda}=(\Phi_{\lambda}, B_{\lambda}, \eta_{\lambda}, h_{\lambda})$ be the one-parameter group associated with $ V$. Then, given
    \begin{equation*}
        P(\lambda)= \frac{d \mathbb{Q}_{\lambda}}{d \mathbb{P}} \cdot \Bigg( \int_0^{f_{- \lambda}(t)} L(F) \circ \Phi_{\lambda}(X_s,s) f'_{\lambda}(s) ds \Bigg),
    \end{equation*}
   we have that 
    \begin{equation*}
        \mathbb{E}_{\mathbb{P}}[|\partial^2_{\lambda} P(\lambda) | ] < \infty.
    \end{equation*}
\end{lemma}

\begin{proof}
 Simply applying the product rule for differentiation and using Theorem \ref{teo ricostruzione del flusso} for the derivative of $\Phi_{\lambda}$ we have that
\small \begin{equation*}
     \partial^2_{\lambda} P(\lambda) = \frac{d \mathbb{Q}_{\lambda}}{d \mathbb{P}} \cdot \Big( \int_0^{f_{-\lambda}(t)} L(F) \circ \Phi_{\lambda}(X_s,s) f'_{\lambda}(s) ds \Big)\Big( \int_0^T \partial_{\lambda} h_{\alpha, \lambda}(X_t,t) dW^{\alpha}_t - \int_0^T h_{\alpha, \lambda}(X_t,t) \partial_{\lambda} h_{\alpha, \lambda}(X_t,t) dt\Big)^2 +
        \end{equation*}
        \begin{equation*}
             + 2 \frac{d \mathbb{Q}_{\lambda}}{d \mathbb{P}} (\partial_{\lambda}f_{-\lambda}(t)) L(F) \circ \Phi_{\lambda}(X_{f_{-\lambda}(t)}, f_{-\lambda}(t)) f'_{\lambda}(f_{-\lambda}(t)) \Big( \int_0^T \partial_{\lambda} h_{\alpha, \lambda}(X_t,t) dW^{\alpha}_t
            - \int_0^T h_{\alpha, \lambda}(X_t,t) \partial_{\lambda} h_{\alpha, \lambda}(X_t,t) dt\Big) + 
        \end{equation*}
        \begin{equation*}
           + 2 \frac{d \mathbb{Q}_{\lambda}}{d \mathbb{P}} \Big( \int_o^{f_{-\lambda}(t)} Y(L(F)) \circ \Phi_{\lambda}(X_s,s) f'_{\lambda}(s) ds \Big) \cdot  \Big( \int_0^T \partial_{\lambda} h_{\alpha, \lambda}(X_t,t) dW^{\alpha}_t - \int_0^T h_{\alpha, \lambda}(X_t,t) \partial_{\lambda} h_{\alpha, \lambda}(X_t,t) dt\Big)+ 
        \end{equation*}
        \begin{equation*}
            +2   \frac{d \mathbb{Q}_{\lambda}}{d \mathbb{P}} \Big( \int_0^{f_{-\lambda}(t)} L(F) \circ \Phi_{\lambda}(X_s,s) \partial_{\lambda} f'_{\lambda}(s) ds \Big) \cdot \Big( \int_0^T \partial_{\lambda} h_{\alpha, \lambda}(X_t,t) dW^{\alpha}_t - \int_0^T h_{\alpha, \lambda}(X_t,t) \partial_{\lambda} h_{\alpha, \lambda}(X_t,t) dt\Big) + 
        \end{equation*}
        \begin{equation*}
           +  \frac{d \mathbb{Q}_{\lambda}}{d \mathbb{P}} \Big( \int_0^{f_{-\lambda}(t)} L(F) \circ \Phi_{\lambda}(X_s,s) f'_{\lambda}(s) ds \Big) \cdot  \Big( \int_0^T \partial^2_{\lambda} h_{\alpha, \lambda}(X_t,t) dW^{\alpha}_t - \int_0^T (\partial_{\lambda} h_{\alpha, \lambda}(X_t,t))^2 dt +
        \end{equation*}
        \begin{equation*}
            - \int_0^T h_{\alpha,\lambda}(X_t,t) \partial^2_{\lambda}h_{\alpha, \lambda}(X_t,t) dt\Big) 
            + 2\frac{d \mathbb{Q}_{\lambda}}{d \mathbb{P}}(\partial_{\lambda} f_{-\lambda}(t)) Y(L(F)) \circ \Phi_{\lambda}(X_{f_{-\lambda}(t)}, f_{-\lambda}(t)) f'_{\lambda}(f_{-\lambda}(t)) + 
        \end{equation*}
        \begin{equation*}
          +  \frac{d \mathbb{Q}_{\lambda}}{d \mathbb{P}}   L(F) \circ \Phi_{\lambda}(X_{f_{-\lambda}(t)}, f_{-\lambda}(t))\partial_{\lambda}(\partial_{\lambda} f_{-\lambda}(t) f'_{\lambda}(f_{-\lambda}(t)) +  \frac{d \mathbb{Q}_{\lambda}}{d \mathbb{P}} \Big(\int_0^{f_{-\lambda}(t)} Y(Y(L(F))) \circ \Phi_{\lambda}(X_s,s) f'_{\lambda}(s) ds \Big) +
        \end{equation*}
        \begin{equation*}
            + 2 \frac{d \mathbb{Q}_{\lambda}}{d \mathbb{P}} \Big(\int_0^{f_{-\lambda}(t)} Y(L(F)) \circ \Phi_{\lambda}(X_s,s) \partial_{\lambda}f'_{\lambda}(s) ds\Big) + \frac{d \mathbb{Q}_{\lambda}}{d \mathbb{P}} \Big( (\partial_{\lambda} f_{-\lambda}(t)) L(F) \circ \Phi_{\lambda}(X_{f_{-\lambda}(t)}, f_{-\lambda}(t)) \partial_{\lambda} f'_{\lambda}(f_{-\lambda}(t) \Big) +
        \end{equation*}
        \begin{equation*}
            + \frac{d \mathbb{Q}_{\lambda}}{d \mathbb{P}} \cdot \Big(  \int_0^{f_{-\lambda}(t)} L(F) \circ \Phi_{\lambda}(X_s,s) \partial^2_{\lambda} f'_{\lambda}(s) ds \Big).
        \ 
\end{equation*}
   \normalsize
\noindent      
    To prove the statement, it suffices to show that each term in the sum lies in \( L^1(\mathbb{P}) \). Only the terms involving the second partial derivative of \( h_{\alpha,\lambda} \)—specifically the fifth—are affected by the Brownian motion rotation, unlike in \cite{paper2023}. We illustrate this by proving the \( L^1 \) membership of a term without such a derivative (requiring no extra assumptions), and then of the term with it, which demands a mild strengthening of the regularity conditions in \cite{paper2023}, forming hypothesis \( A' \).
\\ \\
\noindent
Let us denote with $ (P_{T_{\lambda}}(X),P_{T_{\lambda}}(W))$ the transformed process of $(X,W)$. 
\noindent     We start with the first term: using the Radon-Nikodym theorem, factoring out the common terms, explicitly stating the partial derivative of $ h_{\lambda}$ according to Theorem \ref{teo ricostruzione del flusso}, and applying Definition \ref{trasformate}, we have that
       \begin{equation*}
           \mathbb{E}_{\mathbb{P}}\Bigg[ \frac{d \mathbb{Q}_{\lambda}}{d \mathbb{P}} \cdot \Big( \int_0^{f_{-\lambda}(t)} L(F) \circ \Phi_{\lambda}(X_s,s) f'_{\lambda}(s) ds \Big)\Big( \int_0^T \partial_{\lambda} h_{\alpha, \lambda}(X_t,t) dW^{\alpha}_t - \int_0^T h_{\alpha, \lambda}(X_t,t) \partial_{\lambda} h_{\alpha, \lambda}(X_t,t) dt\Big)^2 \Bigg] =
       \end{equation*}
 \begin{equation*}
     =   \mathbb{E}_{\mathbb{Q}}\Bigg[  \Big( \int_0^{f_{-\lambda}(t)} L(F) \circ \Phi_{\lambda}(X_s,s) f'_{\lambda}(s) ds \Big) \Big( \int_0^T \sqrt{\eta_{\lambda}} B_{\lambda_{\alpha,k}}^{-1} H_{k} \circ \Phi_{\lambda} (X_t,t) \cdot(dW^{\alpha}_t - h_{\alpha, \lambda}(X_t,t) dt ) \Big)^2 \Bigg] .
 \end{equation*}  
 We apply to the first factor the time variable change $\ u=f_{\lambda}(s)$ and we obtain
\begin{equation*}
    \int_0^{f_{-\lambda}(t)} L(F) \circ \Phi_{\lambda}(X_s,s) f'_{\lambda}(s) ds =   \int_0^t L(F) \circ \Phi_{\lambda}(\mathcal{H}_{\eta_{\lambda}}(X)_u du = \int_0^t L(F) (P_{T_{\lambda}}(X))_u  d u.
 \end{equation*}
\noindent
 Similarly, we apply to the second integral the stochastic time change  $ u=f_{\lambda}(t)$. Recalling Lemma \ref{Oksendal} and recalling that $ B \in SO(m)$) and 
 \begin{equation*} 
 dP_{T_{\lambda}}(W_t)=\sqrt{\eta_{\lambda}(f_{-\lambda}(t))}B_{\lambda}(X_{f_{-\lambda}(t)},f_{-\lambda}(t))(dW_{f_{-\lambda}(t)}-h_{\lambda}(X_{f_{-\lambda}(t)},f_{-\lambda}(t)))
 \end{equation*}
 we obtain 
\begin{equation*}
    \int_0^{f_{\lambda}(T)} B_{\lambda_{\alpha,k}}^{-1} H_{k} \circ \Phi_{\lambda} (X_{f_{-\lambda}(u)},f_{-\lambda}(u))  B_{\lambda_{\alpha,k}}^{-1} \cdot dP_{T_{\lambda}}(W^k_u) = 
     \int_0^{f_{\lambda}(T)}   H_{k}   (P_{T_{\lambda}}(X_u))\cdot dP_{T_{\lambda}}(W^k_u)
\end{equation*}

where we used that, since $ B_{\lambda} \in SO(m)$, $ B_{\lambda_{\alpha,i}}^{-1} B_{\lambda_{\alpha,k}}^{-1}= B_{\lambda_{\alpha,i}}^T B_{\lambda_{\alpha,k}}^{-1}= B_{\lambda_{i,\alpha}}B_{\lambda_{\alpha,k}}^{-1}=(B_{\lambda}\cdot B_{\lambda}^{-1})_{i,k}=(I_m)_{ik}=\delta_{i,k}$, where $\delta$ is the Kronecker's delta.
Since $V$ is an infinitesimal symmetry, $X_t$ under $\mathbb{P}$ has the same law as $ (P_{T_{\lambda}}(X_t))$ under $\mathbb{Q}_{\lambda}$, so
\begin{equation*}
     \mathbb{E}_{\mathbb{Q}_{\lambda}}\Bigg[ \Big(\int_0^t L(F (P_{T_{\lambda}}(X_u))) du \Big) \Big( \int_0^{f_{\lambda}(T)}  H_{k}   (P_{T_{\lambda}}(X_u))\cdot dP_{T_{\lambda}}(W^k_u) \Big)^2 \Bigg] =
  \end{equation*}   
   \begin{equation*}  
 =    \mathbb{E}_{\mathbb{P}}\Bigg[ \Big(\int_0^t L(F (X_u))  d u \Big) \Big( \int_0^{T}  H_{k}   (X_u,u)\cdot d W^k_u \Big)^2 \Bigg]. 
\end{equation*}
Applying Holder inequality and It\^o isometry
\begin{equation*}
 \mathbb{E}_{\mathbb{P}}\Bigg[ \Big| \Big(\int_0^t L(F (X_u))  d u \Big) \Big( \int_0^{T} H_{k}   (X_u,u)\cdot d W^k_u \Big)^2 \Big| \Bigg]
  \end{equation*}
\begin{equation*}
  \leq  \Bigg( \mathbb{E}_{\mathbb{P}}\Bigg[  \Bigg(\int_0^t L(F (X_u))  d u \Bigg)^2 \Bigg] \Bigg)^{\frac{1}{2}} \cdot \Bigg( \mathbb{E}_{\mathbb{P}}\Bigg[ \Big( \int_0^{T}   H_{k}   (X_u,u)\cdot d W^k_u \Big)^4\Bigg] \Bigg)^{\frac{1}{2}} 
\end{equation*}
\begin{equation*}
  \leq ||F||_{C^2} \cdot \Bigg( \mathbb{E}_{\mathbb{P}} \Bigg[ \int_0^t |\mu(X_u,u)|^2 du \Bigg] + \mathbb{E}_{\mathbb{P}}\Bigg[ \int_0^t |\sigma(X_u,u)|^2 du \Bigg]\Bigg)^{\frac{1}{2}} \cdot 
\Bigg( \mathbb{E}_{\mathbb{P}}\Bigg[ \Big( \int_0^T \Big|  H_{k}   (X_u,u) \Big|^2 d u \Big)^2\Bigg] .\Bigg)^{\frac{1}{2}} 
\end{equation*}
\noindent
Since $\mu, \sigma,  H_k \in L^2 $ for hypothesis $\ A'$, we just proved that $\partial^2_{\lambda}P(\lambda) \in L^1.$\\
\\
Consider now the fifth addend of $\partial_{\lambda}^2P(\lambda).$ 
Preliminarily, deriving the identity $ B_{\lambda} \cdot B_{\lambda}^{-1} = I_n$ w.r.t. $\lambda$ and applying  Theorem \ref{teo ricostruzione del flusso}, we observe that $   \partial_{\lambda}B_{\lambda}^{-1} = - B_{\lambda}^{-1}  C\circ \Phi_{\lambda}$.
Furthermore, 

\begin{equation*}
    \partial^2_{\lambda} h_{\alpha, \lambda} = \partial_{\lambda} \big( \sqrt{\eta}_{\lambda} B_{\lambda_{\alpha,k} }^{-1}H_{k} \circ \Phi_{\lambda}\big) =  \frac{\tau \circ f_{\lambda}}{2}\sqrt{\eta}_{\lambda} B_{\lambda_{\alpha,k} }^{-1} H_{k} \circ \Phi_{\lambda} - \sqrt{\eta}_{\lambda} (B_{\lambda}^{-1} C )_{\alpha,k} H_{k}\circ \Phi_{\lambda} + \sqrt{\eta_{\lambda}} B_{\lambda_{\alpha,k}}^{-1}Y(H_{k})\circ \Phi_{\lambda}.
\end{equation*}
Let us begin with the second factor of the fifth addend. By appropriately collecting the terms, expanding the derivatives, and applying the time change $ u=f_{\lambda}(t)$, we obtain
 \begin{equation*}
    \int_0^T \partial^2_{\lambda} h_{\alpha,\lambda}(X_t,t) \big( dW^{\alpha}_t-h_{\alpha,\lambda}(X_t,t)dt \big) - \int_0^T (\partial_{\lambda} h_{\alpha, \lambda}(X_t,t))^2 dt=
\end{equation*}
\begin{equation*}
     = \int_0^T \big( \frac{\tau(f_{\lambda}(t))}{2}B^{-1}_{\lambda_{\alpha,k}}H_k\circ \Phi_{\lambda}(X_t,t)-(B_{\lambda}^{-1} C)_{\alpha,k} H_k \circ \Phi_{\lambda}(X_t,t) + 
\end{equation*}
\begin{equation*}
   +B^{-1}_{\lambda_{\alpha,k}}Y(H_k)\circ \Phi_{\lambda}(X_t,t)\big)\sqrt{\eta_{\lambda}(t)} \cdot \Big( (B^{-1}_{\lambda} \cdot B_{\lambda})_{\alpha,j}\big( dW^{j}_t-h_{j,\lambda}(X_t,t)dt \big)\Big)  - \int_0^T \eta_{\lambda}(t) \big( B^{-1}_{\lambda_{\alpha,k}}H_k\circ\Phi_{\lambda}(X_t,t) \big)^2 dt=
\end{equation*}
\begin{equation*}
     = \int_0^T \big( \frac{\tau(f_{\lambda}(t))}{2}(B_{\lambda} B^{-1}_{\lambda})_{k,l} H_k\circ \Phi_{\lambda}(X_t,t)- (B_{\lambda}  B^{-1}_{\lambda})_{i,l} C_{i,k} H_k \circ \Phi_{\lambda}(X_t,t) + 
\end{equation*}
\begin{equation*}
   +(B_{\lambda} B^{-1}_{\lambda})_{k,l} Y(H_k)\circ \Phi_{\lambda}(X_t,t)\big)\sqrt{\eta_{\lambda}(t)}  B_{\lambda_{l,j}}\big( dW^{j}_t-h_{j,\lambda}(X_t,t)dt \big)  - \int_0^T \eta_{\lambda}(t) \big( B^{-1}_{\lambda_{\alpha,k}}H_k\circ\Phi_{\lambda}(X_t,t) \big)^2 dt=
\end{equation*}
\begin{equation*}
     = \int_0^T \big( \frac{\tau(f_{\lambda}(t))}{2}H_l\circ \Phi_{\lambda}(X_t,t)-  C_{l,k} H_k \circ \Phi_{\lambda}(X_t,t) + 
\end{equation*}
\begin{equation*}
   + Y(H_l)\circ \Phi_{\lambda}(X_t,t)\big)\sqrt{\eta_{\lambda}(t)}  B_{\lambda_{l,j}}\big( dW^{j}_t-h_{j,\lambda}(X_t,t)dt \big)  - \int_0^T \eta_{\lambda}(t) \big( B^{-1}_{\lambda_{\alpha,k}}H_k\circ\Phi_{\lambda}(X_t,t) \big)^2 dt,
\end{equation*}
where, once again, we used that $ B_{\lambda_{\alpha,i}}^{-1} B_{\lambda_{\alpha,k}}^{-1}=\delta_{i,k}$.
Applying the time change $\  t=f_{-\lambda}(s)$ the expectation under $\mathbb{P}$ of the fifth addend becomes
\small \begin{equation*}
    \mathbb{E}_{\mathbb{P}}\Bigg[\frac{d \mathbb{Q}_{\lambda}}{d \mathbb{P}} \Big( \int_0^{f_{-\lambda}(t)} L(F) \circ \Phi_{\lambda}(X_s,s) f'_{\lambda}(s) ds \Big) \cdot  \Bigg( \int_0^T \partial^2_{\lambda} h_{\alpha, \lambda}(X_t,t) (dW^{\alpha}_t-h_{\alpha,\lambda}(X_t,t)dt)  - \int_0^T (\partial_{\lambda} h_{\alpha, \lambda}(X_t,t))^2 dt\Bigg)\Bigg]
\end{equation*}
\begin{equation*}
=\mathbb{E}_{\mathbb{Q_{\lambda}}}\Bigg[ \Big( \int_0^t L(F) (P_{T_{\lambda}}(X_s)) ds \Big) \cdot \Bigg( \int_0^{f_{\lambda}(T)} \frac{\tau}{2} H_l(P_{T_{\lambda}}(X_t),t) -C_{l,k}H_k(P_{T_{\lambda}}(X_t),t)+
\end{equation*}
\begin{equation*}
  + Y(H_l) (P_{T_{\lambda}}(X_t),t)dP_{T_{\lambda}}(W_t)^l  - \int_0^{f_{\lambda}(T)} ( B^{-1}_{\lambda_{\alpha,k}})^2 H^2_k(P_{T_{\lambda}}(X_t),t) ds\Bigg)\Bigg]=
\end{equation*}
\begin{equation*}
=\mathbb{E}_{\mathbb{P}}\Bigg[ \Big( \int_0^t L(F (X_s)) ds \Big) \cdot \Bigg( \int_0^{f_{\lambda}(T)} \frac{\tau}{2} H_l(X_s,s) -C_{l,k}H_k(X_s,s)+  Y(H_l) (X_s,s)dW_s^l  - \int_0^{f_{\lambda}(T)} H^2_k(X_s,s) ds \Bigg)\Bigg]
\end{equation*}\normalsize
where we used again the fact that $X_t,W_t $ under $\mathbb{P}$ has the same law of $ P_{T_{\lambda}}(X_t,W_t)$ under $\mathbb{Q}_{\lambda}$ and that $ B_{\lambda}B_{\lambda}^T=I_m$. Consider the second term, which is the only one affected by the presence of rotation. By using Hölder inequality and Itô isometry
\footnotesize\begin{equation*}
  \mathbb{E}_{\mathbb{P}}\Bigg[ \Bigg|\int_0^t L(F (X_s)) ds  \Bigg| \cdot \Bigg|\int_0^T C_{l,k} H_{k}(X_t,t) \cdot dW_t\Bigg|\Bigg] 
    \leq \Bigg(\mathbb{E}_{\mathbb{P}}\Bigg[ \Bigg(\int_0^t L(F (X_s)) ds  \Bigg)^2 \Bigg] \Bigg)^{\frac{1}{2}} \Bigg(\mathbb{E}_{\mathbb{P}}\Bigg[ \Bigg(  \int_0^T   C_{l,k} H_{k}(X_t,t) \cdot dW_t\Bigg)^2 \Bigg] \Bigg)^{\frac{1}{2}} 
\end{equation*}
\begin{equation*}
  \leq  ||F||_{C^2} \Bigg(\mathbb{E}_{\mathbb{P}}\Bigg[ \int_0^t |\mu(X_s,s)|^2 ds + \int_0^t |\sigma(X_s,s)|^2ds \Bigg] \Bigg)^{\frac{1}{2}} \Bigg(\mathbb{E}_{\mathbb{P}}\Bigg[  \int_0^{f_{\lambda(T)}} \Big| C_{
    l,k} H_{k}(X_t,t)\Big|^2 dt  \Bigg] \Bigg)^{\frac{1}{2}} 
\end{equation*}\normalsize
and this term is finite if, as assumed in hypothesis A', $\mu, \sigma, C_{l,k}H_k \in L^2.$\\
 All other terms can be treated similarly.
\end{proof}

\section{A rotational invariant integration by parts formula}\label{Rotational invariante formula}
In this work, we extend the integration by parts formula derived in \cite{paper2023} by including also rotations of Brownian motion. Notably, this extension leaves the structure of the formula unchanged, while the analytical results in Section \ref{analytical considerations} highlight the refinements required for a rigorous formulation.
\begin{teorema}[integration by parts formula]\label{teo integrazione per parti 2}
Let $ (X,W)$ be a solution to  $ SDE_{\mu,\sigma}$ and let $ V= (Y, C, \tau, H)$ be an infinitesimal stochastic symmetry for  $SDE{\mu,\sigma}$. Let $T_{\lambda}=(\Phi_{\lambda}, B_{\lambda}, \eta_{\lambda}, h_{\lambda})$ be the one-paramer group related to $ V$, where, as usual $\Phi_{\lambda}=(\tilde{\Phi}_{\lambda},f_{\lambda})$ and $\ Y=(\tilde{Y},m)$. Assuming Hypothesis A' (\ref{ipotesi A'}), then the following integration by parts formula holds for every $ t \in [0,T]$, for all $ F$  bounded functional with bounded first and second derivative:
\begin{equation}\label{integrazione per parti 2}
    -m(t) \mathbb{E}_{\mathbb{P}}[L(F(X_t))] + \mathbb{E}_{\mathbb{P}}[F(X_t) \int_0^t (H_{\alpha}(X_s,s) d W_s^{\alpha})] + \mathbb{E}_{\mathbb{P}}[Y(F(X_t))] - \mathbb{E}_{\mathbb{P}}[Y(F(X_0))]=0.
\end{equation}
\end{teorema}

\begin{proof}
Let $ P_{T_\lambda}(X,W):= (P_{T_{\lambda}}(X), P_{T_{\lambda}}(W)) $ be the solution process transformed by the one-parameter group $ T_{\lambda}$.
    Since $ V$ is an infinitesimal symmetry,  the one-parameter group $T_{\lambda}$ generated by $V$ is a finite stochastic symmetry, so, from Quasi invariance principle (\ref{quasi-invarianza})
    \begin{equation*}
        \mathbb{E}_{\mathbb{P}}[F(X_t)]= \mathbb{E}_{\mathbb{Q_{\lambda}}}[F(P_{T_{\lambda}}(X_t))] \ \ \ \ \forall t \in [0, T],
    \end{equation*}
    where $\mathbb{Q_{\lambda}}$ is the new probability measure obtained by Girsanov theorem via $ T_{\lambda}$.
    Applying It\^o formula and recalling the definition of transformed process we have that    
    \begin{equation*}
        \mathbb{E}_{\mathbb{Q_{\lambda}}}[F(P_{T_{\lambda}}(X_t))] = \mathbb{E}_{\mathbb{Q_{\lambda}}} \Bigg[\int_0^t L(F([P_{T_{\lambda}}(X)]_s)) ds + \partial_i (F)([P_{T_{\lambda}}(X)]_s) \sigma^i_{\alpha} d[P_{T_{\lambda}}(W^{\alpha})]_s\Bigg] = 
    \end{equation*}
    \begin{equation}\label{dim sparisce Wl}
        =\mathbb{E}_{\mathbb{Q_{\lambda}}} \Bigg[\int_0^t L(F([P_{T_{\lambda}}(X)]_s)) ds \Bigg] = \mathbb{E}_{\mathbb{Q_{\lambda}}} \Bigg[\int_0^t L(F(\Phi_{\lambda}(X_{f_{-\lambda}(s)}))) ds\Bigg],
    \end{equation}
    where the penultimate equality is justified by the martingale property of the stochastic integral. In fact, given $ g(s,X_s)=\partial_i (F \circ \Phi_{\lambda})(X_s) \sigma^i_{\alpha}$, if $ g$ is $\mathcal{B} \otimes \mathcal{F}$-measurable, if $\forall s$, $g(s,.) $ is $\mathcal{F}$-measurable and, $\forall s, g(s,.) \in L^2(\Omega, \mathcal{F}, \mathbb{Q}_{\lambda})$, $\mathbb{E}_{\mathbb{Q}_{\lambda}}\Bigg[\int_0^t |g(s,.)|^2 ds \Bigg]<\infty,$ then the stochastic integral with respect to the $\mathbb{Q}_{\lambda}$ -Brownian motion $ P_{T_{\lambda}}(W)$ is a (global) martingale, and therefore has a constant expected value equal to the initial value (which is zero). Hypothesis A' ensures that the integrand is in $\ L^2$  so that It\^o integral is well defined.\\
   Let \( f_{-\lambda} \) denote the inverse of \( f_{\lambda} \). Applying the change of variables \(  u = f_{-\lambda}(s) \)  and explicitly using the Radon-Nikodym derivative provided by Girsanov theorem (see Section \ref{Finite_stochastic_transformations_SDE}), we obtain
\begin{equation*}
      \mathbb{E}_{\mathbb{Q_{\lambda}}} \Bigg[\int_0^t L(F(\Phi_{\lambda}(X_{f_{-\lambda}(s)}))) ds\Bigg]= \mathbb{E}_{\mathbb{P}}\Bigg[\frac{ d\mathbb{Q}_{\lambda}}{d \mathbb{P}} \int_0^{f_{-\lambda}(t)}L(F(\Phi_{\lambda}(X_u))) f'_{\lambda}(u )))du  \Bigg].
\end{equation*}
Retracing the chain of equalities, we thus obtain that 
\begin{equation}\label{uguaglianza valori attesi}
    \mathbb{E}_{\mathbb{P}}[F(X_t)] = \mathbb{E}_{\mathbb{P}}\Bigg[\frac{ d \mathbb{Q}_{\lambda}}{d \mathbb{P}}\int_0^{f_{-\lambda}(t)}L(F(\Phi_{\lambda}(X_u))) f'_{\lambda}(u )du  \Bigg].
\end{equation}
Denoting by $ P(\lambda)$ the argument of the second expectation $    P(\lambda)=Z_{\lambda} \int_0^{f_{-\lambda}(t)}L(F(\Phi_{\lambda}(X_s)))f'_{\lambda}(s) ds$ and with $ Z_{\lambda}=\frac{ d\mathbb{Q}_{\lambda}}{d\mathbb{P}}$, Hypothesis A', jointly with Lemma \ref{lemma tecnico mb} ensure $\forall \lambda  \ \  \partial^2_{\lambda} P(\lambda) \in L^2,$ and therefore, from Lemma \ref{Kunita}, $    \partial_{\lambda} \mathbb{E}_{\mathbb{P}}[P(\lambda)]=\mathbb{E}_{\mathbb{P}}[\partial_{\lambda} P(\lambda)].$
Differentiating both sides of equation \eqref{uguaglianza valori attesi} with respect to the parameter $\lambda$ and evaluating at $\lambda=0$, we obtain 
\begin{equation}\label{5.7}
0 = \mathbb{E}_{\mathbb{P}} \Bigg[ \Bigg(\int_0^T H_{\alpha}(X_t,t) dW_t^{\alpha} \Bigg)\Bigg(\int_0^t L(F(X_s)) ds\Bigg)\Bigg] - m(t) \mathbb{E}_{\mathbb{P}}[L(F(X_t))] + 
\end{equation}
\begin{equation*}
    +\mathbb{E}_{\mathbb{P}}\Bigg[\int_0^t Y(L(F(X_s)))ds\Bigg]+\mathbb{E}_{\mathbb{P}}\Bigg[\int_0^t \tau(s) L(F(X_s)) ds\Bigg]
\end{equation*}
where the first term comes from the derivative of $ Z_{\lambda}$, recalling that (c.f. Theorem \ref{teo ricostruzione del flusso}) 
\begin{equation}\label{derivata h}
\partial_{\lambda}h_{\alpha,\lambda}|_{\lambda=0}=\big(\frac{1}{\sqrt{\eta}_{\lambda}}B_{\lambda}^{-1}H_{\alpha}\circ\Phi_{\lambda}\big)_{|_{\lambda=0}}=\frac{1}{\sqrt{\eta}_{0}}B_0^{-1}H_{\alpha}\circ\Phi_0=H_{\alpha};
\end{equation}
the second term is obtained by applying the fundamental theorem of calculus 
and from the contribution of the derivative w.r.t. $\lambda$ of the integration limits, considering that $\partial_{\lambda} f_{-\lambda}(t) _{|_{\lambda=0}} = - \partial_{\lambda} f_{\lambda}(t) _{|_{\lambda=0}} = - m(t) $ and that $\ f'_{\lambda}(t)=:\eta_{\lambda}(t)$; the third addend is due to the derivative w.r.t. $\lambda$ of $\Phi_{\lambda}$ and recalling that $\Phi_0=Id$; finally the last term is obtained by observing that $\partial_{\lambda} f'_{\lambda} = \partial_{\lambda} \eta_{\lambda} = \tau \circ \Phi_{\lambda} \cdot \eta_{\lambda}$, e $\eta_0=1$ e $\Phi_0=Id.$ Since for an infinitesimal symmetry $ [Y,L] = - \tau L + H\sigma D$, we have that
\begin{equation*}
    \mathbb{E}_{\mathbb{P}}[\int_0^t Y(L(F(X_s)))ds] =\mathbb{E}_{\mathbb{P}}[\int_0^t L(Y(F(X_s)))ds] + [Y,L] = 
\end{equation*}
\begin{equation*}
  =\mathbb{E}_{\mathbb{P}}[\int_0^t L(Y(F(X_s)))ds] - \mathbb{E}_{\mathbb{P}}[\int_0^t \tau(s) L(F(X_s))ds]+   \mathbb{E}_{\mathbb{P}}[\int_0^t H_{\alpha}(X_s,s) \sigma_{\alpha} D F (X_s) ds] .
\end{equation*}
By It\^o formula we have
\begin{equation*}
    \mathbb{E}_{\mathbb{P}}[Y(F(X_t))]= \mathbb{E}_{\mathbb{P}}[Y(F(X_0))] + \mathbb{E}_{\mathbb{P}}[\int_0^t L(Y(F(X_s)))ds] + \mathbb{E}_{\mathbb{P}}[\int_0^t \partial_iY(F(X_s)) \sigma_{\alpha}^i dW_s^{\alpha}] =
\end{equation*}
\begin{equation*}
   = \mathbb{E}_{\mathbb{P}}[Y(F(X_0))] + \mathbb{E}_{\mathbb{P}}[\int_0^t L(Y(F(X_s)))ds],
\end{equation*}
from which we obtain$\ 
    \mathbb{E}_{\mathbb{P}}[\int_0^t L(Y(F(X_s)))ds] = \mathbb{E}_{\mathbb{P}}[Y(F(X_t))] - \mathbb{E}_{\mathbb{P}}[Y(F(X_0))].$
Thus (\ref{5.7}) becomes
\begin{equation}\label{5.7 finale}
    0 = \mathbb{E}_{\mathbb{P}} \Bigg[ \left(\int_0^T H_{\alpha}(X_t,t) dW_t^{\alpha} \right)\left(\int_0^t L(F(X_s))) ds\right)\Bigg] - m(t) \mathbb{E}_{\mathbb{P}}[L(F(X_t))] +
\end{equation}
\begin{equation*}
    + \mathbb{E}_{\mathbb{P}}[Y(F(X_t))] - \mathbb{E}_{\mathbb{P}}[Y(F(X_0))] 
 + \mathbb{E}_{\mathbb{P}}\Bigg[\int_0^t H_{\alpha}(X_s,s) \sigma_{\alpha} D F (X_s) ds](X_s)ds\Bigg].
\end{equation*}

Let us consider the first addend of (\ref{5.7 finale}). Since stochastic integrals are martingales we get
\begin{equation*}
    \mathbb{E}_{\mathbb{P}} \Bigg[ \left(\int_0^T H_{\alpha}(X_t,t) dW_t^{\alpha} \right)\left(\int_0^t L(F(X_s)) ds\right)\Bigg] = \mathbb{E}_{\mathbb{P}} \Bigg[ \left(\int_0^t H_{\alpha}(X_t,t) dW_t^{\alpha} \right)\left(\int_0^t L(F(X_s)) ds\right)\Bigg]
\end{equation*}
and integrating by parts we have
\begin{equation*}
     \left(\int_0^t H_{\alpha}(X_t,t) dW_t^{\alpha} \right)\left(\int_0^t L(F(X_s)) ds\right) = \int_0^t \left( \int_0^s H_{\alpha}(X_{\tau},\tau)dW_{\tau}^{\alpha}\right) L(F(X_s))ds + 
\end{equation*}
\begin{equation*}
    +\int_0^t \left( \int_0^s L(F(X_{\tau}))d\tau\right) H_{\alpha}(X_s,s) dW_s^{\alpha} + \bcancel{\left[\int_0^t H_{\alpha}(X_s,s) dW_s^{\alpha}, \int_0^t L(F(X_s))ds\right]} 
\end{equation*}
where the last quadratic covariation is zero because the second term is absolutely continuous.
Therefore the first addend of (\ref{5.7 finale}) becomes
\begin{equation*}
      \mathbb{E}_{\mathbb{P}} \Bigg[ \left(\int_0^T H_{\alpha}(X_t,t) dW_t^{\alpha} \right)\left(\int_0^t L(F(X_s)) ds\right)\Bigg] =    \mathbb{E}_{\mathbb{P}} \Bigg[ \left(\int_0^t H_{\alpha}(X_t,t) dW_t^{\alpha} \right)\left(\int_0^t L(F(X_s)) ds\right)\Bigg]=
\end{equation*}
\begin{equation*}
    = \mathbb{E}_{\mathbb{P}}\Bigg[ \int_0^t \left( \int_0^s H_{\alpha}(X_{\tau},\tau)dW_{\tau}^{\alpha}\right) L(F(X_s))ds\Bigg] +\mathbb{E}_{\mathbb{P}}\Bigg[\int_0^t \left( \int_0^s L(F(X_{\tau}))d\tau\right) H_{\alpha}(X_s,s) dW_s^{\alpha}\Bigg]=: \textbf{(A)}+\textbf{(B)}. 
\end{equation*}
Again, since stochastic integrals are martingales we have
\begin{equation*}
\textbf{(A)}=     \mathbb{E}_{\mathbb{P}}\Bigg[\int_0^t \left( \int_0^s H_{\alpha}(X_{\tau},\tau) dW_{\tau}^{\alpha}\right) L(F(X_s)) ds + \int_0^t\left(\int_0^s H_{\alpha}(X_{\tau},\tau)dW_{\tau}^{\alpha}\right) \sigma_{\alpha} D F(X_s) dW_s^{\alpha} \Bigg]
\end{equation*}
and integrating by parts the second addend we obtain 
\begin{equation*}
\int_0^t \left(\int_0^s H_{\alpha}(X_\tau,\tau) dW_{\tau}^{\alpha}\right) \sigma_{\alpha} D F(X_s) dW_s^{\alpha} = \left(\int_0^t H_{\alpha}(X_s,s) dW_s^{\alpha}\right)\left(\int_0^t \sigma_{\alpha} DF (X_s) dW_s^{\alpha}\right)+
\end{equation*}
\begin{equation*}
     - \int_0^t \left( \int_0^s \sigma_{\alpha} D F (X_{\tau}) d W_{\tau}^{\alpha}\right) H_{\alpha}(X_s,s) dW_s^{\alpha} - \Big[\int_0^t H_{\alpha}(X_s,s) dW_s^{\alpha}, \int_0^t \sigma_{\alpha} D F(X_s) dW_s^{\alpha}\Big]
\end{equation*}
with the quadratic variation equal to
\begin{equation*}
    \Big[\int_0^t H_{\alpha}(X_s,s) dW_s^{\alpha}, \int_0^t \sigma_{\alpha} D F(X_s) dW_s^{\alpha}\Big]= \int_0^t \sigma_{\alpha} D F(X_s) \cdot H_{\alpha}(X_s,s).
\end{equation*}
So 
\small \begin{equation*}
  \textbf{(A)}+ \textbf{(B)} =   \mathbb{E}_{\mathbb{P}}\Big[\int_0^t  \left(\int_0^s H_{\alpha}(X_{\tau},\tau) dW_{\tau}^{\alpha}\right) L(F(X_s)) ds\Big] + \mathbb{E}_{\mathbb{P}}\Big[ \left(\int_0^t H_{\alpha}(X_s,s) dW_s^{\alpha}\right)\left(\int_0^t \sigma_{\alpha} D F (X_s) dW_s^{\alpha}\right)\Big] \end{equation*}
\small\begin{equation*} - \mathbb{E}_{\mathbb{P}}\Big[\int_0^t \left( \int_0^s \sigma_{\alpha} D F (X_{\tau}) d W_{\tau}^{\alpha}\right) H_{\alpha}(X_s,s) dW_s^{\alpha}\Big]  - \mathbb{E}_{\mathbb{P}}\Big[\int_0^t \sigma_{\alpha} D F(X_s) \cdot H_{\alpha}(X_s,s)\Big] +
\end{equation*}
\small\begin{equation*}
+\mathbb{E}_{\mathbb{P}}\Big[\int_0^t \left( \int_0^s L(F(X_{\tau}))d\tau\right) H_{\alpha}(X_s,s) dW_s^{\alpha}\Big],
\end{equation*}
\normalsize i.e., by factoring, by martingality of stochastic integrals and by It\^o formula the first term of \eqref{5.7 finale} becomes
\begin{equation*}
    \mathbb{E}_{\mathbb{P}}\Big[\left(\int_0^t L(F(X_s))ds + \int_0^t \sigma_{\alpha} D F(X_s) dW_s^{\alpha}\right) \left(\int_0^t H_{\alpha}(X_s,s) dW_s^{\alpha}\right)\Big]
\end{equation*}
\begin{equation*}
    - \bcancel{\mathbb{E}_{\mathbb{P}} \Bigg[\int_0^t \Bigg(\int_0^s L(F(X_{\tau}))d\tau + \int_0^s \sigma_{\alpha} D F(X_{\tau}) dW_{\tau}^{\alpha}\Bigg)H_{\alpha}(X_s,s) dW_s^{\alpha}\Bigg]
 } + 
\end{equation*}
\begin{equation*}
     - \mathbb{E}_{\mathbb{P}}\Big[\int_0^t \sigma_{\alpha} D F(X_s) H_{\alpha}(X_s,s) ds\Big] =   \mathbb{E}_{\mathbb{P}}[F(X_t)\int_0^t H_{\alpha}(X_s,s) dW_s^{\alpha}]  - \mathbb{E}_{\mathbb{P}}\Big[\int_0^t \sigma_{\alpha} DF(X_s) H_{\alpha}(X_s,s) ds\Big] .
\end{equation*}
Hence, substituting the expression found for the first addend \eqref{5.7 finale} becomes
\begin{equation*}
    0 =  \mathbb{E}_{\mathbb{P}}\Big[F(X_t)\int_0^t H_{\alpha}(X_s,s) dW_s^{\alpha}\Big]  - \bcancel{\mathbb{E}_{\mathbb{P}}\Big[\int_0^t \sigma_{\alpha} D F(X_s) H_{\alpha}(X_s,s) ds\Big]}  - m(t) \mathbb{E}_{\mathbb{P}}[L(F(X_t))] +
\end{equation*}
\begin{equation*}
    +\mathbb{E}_{\mathbb{P}}[Y(F(X_t))] - \mathbb{E}_{\mathbb{P}}[Y(F(X_0))] 
 +  \bcancel{\mathbb{E}_{\mathbb{P}}\Big[\int_0^t H_{\alpha}(X_s,s) \sigma_{\alpha} D F (X_s) ds\Big]}
\end{equation*}
and the proof is complete.
\end{proof}

\begin{osservazione}
The proof mirrors \textit{Theorem 28 in \cite{paper2023}}, which addresses the case without rotation. The only differences arise from the rotation matrix \( B_{\lambda} \), which appears in the transformed Brownian motion \( P_{T_{\lambda}}(W) \) and affects certain derivatives in \( T_{\lambda} \) (see (\ref{derivata h})). In contrast, \( P_{T_{\lambda}}(X) \) and the Radon-Nikodym derivative remain unchanged. Additionally, the Brownian component vanishes due to the martingale property of the stochastic integral, and since \( B_0 = I_m \), the derivative of the density \( Z_{\lambda} \) at \( \lambda = 0 \) coincides with the non-rotated case.
\end{osservazione}

\begin{osservazione}\label{remark generatore infinitesimale}
 The invariance of the integration by parts formula under a rotation of Brownian motion is linked to the well-known invariance of the infinitesimal generator of an 
$ SDE_{\mu,\sigma} $ under the same transformation. 
Indeed, this formula is obtained by taking advantage of the invariance of the solution's law, and since the process law is determined by its generator \( L \), a rotation of Brownian motion preserves both the solution law and the structure of the integration by parts formula.
\end{osservazione}

\section{Examples}\label{Examples}
\subsection*{Two dimensional Brownian motion}
Consider the  SDE
\begin{equation}\label{BM2}
    \begin{pmatrix}
        dX_t\\
        dY_t\\
        dZ_t
    \end{pmatrix} = \underbrace{\begin{pmatrix}
        0\\
        0\\
        1
    \end{pmatrix}}_{\mu}dt + \underbrace{\begin{pmatrix}
        1 & 0 \\
        0 & 1 \\
        0 & 0
    \end{pmatrix}}_{\sigma} \begin{pmatrix}
        dW^1_t\\
        dW^2_t
    \end{pmatrix}
\end{equation}
corresponding to a two-dimensional Brownian motion, where the third component, admitting solution $ Z_t = t$, has been introduced in order to include in our setting also time dependent transformations.\\
We are looking for infinitesimal symmetries $ V=(Y,C,\tau,H)$, with
\begin{equation*}
\begin{pmatrix}
   Y= \begin{pmatrix}
        f(x,y,z)\\
        g(x,y,z)\\
        m(x,y,z)
    \end{pmatrix};
 &

    C= \begin{pmatrix}
        0 & c(x,y,z) \\
        -c(x,y,z) & 0
    \end{pmatrix};
    &   \tau=\tau(x,y,z);

    &   H=\begin{pmatrix}
        H_1(x,y,z)\\
        H_2(x,y,z)
    \end{pmatrix}
\end{pmatrix}
\end{equation*} 
satisfying the determining equations from Theorem \ref{equazioni determinanti}.
From the first determining equation in \eqref{equazione determinante 1.1} we have 
\begin{equation*}
    L(Y)-Y(\mu) + \sigma H =\tau \mu \underbrace{\Rightarrow}_{Y(\mu)=0}
  \underbrace{ [ \frac{1}{2} \partial_{xx} + \frac{1}{2} \partial_{yy}  + 1 \partial_z]}_{L=A^{ij}\partial_{ij}+\mu^i\partial_i} \underbrace{\begin{pmatrix}
       f\\
       g\\
       m
   \end{pmatrix}}_{Y} + \underbrace{\begin{pmatrix}
      1 & 0 \\
        0 & 1 \\
        0& 0
   \end{pmatrix}}_{\sigma } \underbrace{\begin{pmatrix}
       H_1\\
       H_2\\
   \end{pmatrix}}_{H} = \tau \underbrace{\begin{pmatrix}
       0\\
       0\\
       1
   \end{pmatrix}}_{\mu}
\end{equation*}
$\ i.e.$ 
\begin{equation*}
    \begin{matrix}
        \frac{1}{2} f_{xx}+\frac{1}{2}f_{yy} + f_z + k_x=0; &
     \frac{1}{2} g_{xx}+\frac{1}{2}g_{yy} + g_z + k_y=0; &
     \frac{1}{2} m_{xx}+\frac{1}{2}m_{yy} + m_z =\tau\\
    \end{matrix}.
\end{equation*}
From the second determining equation we have
\begin{equation*}
    [Y,\sigma] + \frac{1}{2}\tau\sigma+\sigma C=0 \underbrace{\Rightarrow}_{Y(\sigma)=0} \underbrace{- \begin{pmatrix}
        f_x & f_y & f_z \\
        g_x & g_y & g_z \\
        m_x & m_y & m_z 
    \end{pmatrix}\begin{pmatrix}
    1 & 0\\
    0& 1 \\
    0 & 0
    \end{pmatrix}}_{\sigma(Y)= D(Y) \sigma} + \underbrace{\begin{pmatrix}
        \frac{1}{2}\tau & 0\\
        0 & \frac{1}{2} \tau\\
        0 & 0
    \end{pmatrix}}_{\frac{1}{2}\tau \sigma} + \underbrace{\begin{pmatrix}
        0 & c\\
        -c & 0\\
        0 & 0
    \end{pmatrix}}_{\sigma C}=0
\end{equation*}
$\ i.e.$
\begin{equation*}
    \begin{matrix}
        f_x=\frac{1}{2}\tau; &
        f_y=c; &
        g_x=-c; &
        g_y=\frac{1}{2} \tau; &
        m_x=0;&
        m_y=0.
    \end{matrix}
\end{equation*}
So we must solve
\begin{equation*}
    \begin{matrix}
        \frac{1}{2} f_{xx}+\frac{1}{2}f_{yy}+f_z+ H_1=0, \ \ &
        \frac{1}{2} g_{xx}+\frac{1}{2}g_{yy}+g_z+H_2=0,\ \ &
        \frac{1}{2} m_{xx}+\frac{1}{2}m_{yy}+m_z = \tau, \\
        f_x=\frac{1}{2}\tau,\ \ \ \ \ \  \  \  \ \ \
        f_y=c,&
        g_x=-c, \ \ \ \  \ \ \ \ \ 
        g_y=\frac{1}{2}\tau,&
        m_x=0,\ \ \ \ \ \ \ \ \ \ 
        m_y=0.
    \end{matrix}
\end{equation*}
Since there are no conditions on the components of the vector field $ H$, we can find two families of infinite-dimensional symmetries, each of them depending on an arbitrary function of $ z$ (respectively, $\alpha(z)$ and $\beta(z)$).\\
\noindent In fact, from the last two equations $ m_x=m_y=0$ we obtain $ m(x,y,z)=m(z),$ from which two choices are possible: $m$ is an arbitrary function of $ z$, or $ m \equiv 0$. The first family of symmetries $V_{\alpha}$ is obtained by choosing
 $ m=m(z)$ , so
\begin{equation}\label{Va}
V_{\alpha}= \Bigg( \begin{pmatrix}
    \frac{1}{2}\alpha(z)x \\ \frac{1}{2}\alpha(z) y \\ \int \alpha(z) dz
\end{pmatrix}, \begin{pmatrix}
    0 & 0\\ 0 & 0
\end{pmatrix}, \alpha(z), \begin{pmatrix}
    -\frac{1}{2} x \alpha'(z)\\ -\frac{1}{2}y \alpha'(z)
\end{pmatrix}\Bigg) .   
\end{equation}
The second family is obtained by choosing $ m=0$:
\begin{equation}\label{Vb}
    V_{\beta}= \Bigg( \begin{pmatrix}
        \beta(z)y\\-\beta(z)x\\0
    \end{pmatrix}, \begin{pmatrix}
        0 & \beta(z) \\ - \beta(z) & 0 
    \end{pmatrix}, 0, \begin{pmatrix}
        -y\beta'(z)\\ x\beta'(z)
    \end{pmatrix}
    \Bigg).
\end{equation}

\begin{osservazione}\label{invariance mb}
Symmetries $V_{\alpha}$ are related to the random time invariance of Brownian 
motion,
while symmetries $V_{\beta}$ are related to the invariance of Brownian motion with respect to random rotation (see Section \ref{Random rotation of Brownian motion}). Indeed,  given $ B(t)$ a rotation matrix depending on time $ t$, while Proposition \ref{invarianza per rotazioni} ensures that $ W'_t=\int_0^t B(t) dW_t$ is still a $\mathbb{P}-$Brownian motion, the symmetry $V_{\beta}$ provides a generalization of this classical result, considering instead of $W'$ the process $\tilde{W'}_t= B(t)  W_t.$ Since $W'_{t'}$ is already a $\mathbb{P}$-Brownian motion, $\tilde{W'}$ cannot be himself a $\mathbb{P}$-Brownian motion, but is still a Brownian motion with respect to a new probability measure $\mathbb{Q}.$ This is well known from classical It\^o's calculus, since \[ d\tilde{W'}_t=d (B(t)W_t)=B(t)dW_t+B'(t)W_t dt=B(t)dW_t-(-B'(t)W_t) dt=dW'_t-(-B'(t)W_t)dt\] and therefore, $d\tilde{W'}$
  is composed of a Brownian component $dW$ from which a drift term is subtracted, namely $h(t,X_t)=h(t,W_t)=B'(t)W_t$ (remember that in this example the solution $ X$ is the two dimensional Brownian motion $W$). By Girsanov’s theorem, $\tilde{W'}$ is a Brownian motion under the new probability measure $\mathbb{Q}$, whose Radon-Nikodym density is given by \footnotesize$\ \frac{d\mathbb{Q}}{d\mathbb{P}}_{|_{\mathcal{F_t}}}= \exp(\int_0^t h(s,X_s) dW_s-\frac{1}{2}\int_0^t |h(s,X_s)|^2ds)= \exp\Bigg( - \int_0^t  B'(s) W_s \cdot dW_s - \frac{1}{2} \int_0^T | B'(s) \cdot W_s|^2 ds \Bigg).$ \normalsize  Symmetry $ V_{\beta}$ encodes this invariance property. Indeed, from Theorem \ref{teo ricostruzione del flusso}, we know that $V_{\beta}$ is related to a finite transformation $ T_{\lambda}=(\Phi_{\lambda}, B_{\lambda}, \eta_{\lambda}, h_{\lambda})$ where\footnote{From flow reconstruction it is sufficient to solve the following ODEs:\\ \scriptsize $\frac{d x_{\lambda}}{d\lambda}=\beta(z_{\lambda})y, \frac{d y_{\lambda}}{d\lambda}=-\beta(z_{\lambda})x, \frac{dz_{\lambda}}{d\lambda}=0, \frac{d B_{\lambda}}{d\lambda}=\begin{pmatrix}
      0&\beta(z_{\lambda}\\-\beta(z_{\lambda}&0
  \end{pmatrix}B_{\lambda}, \frac{d\eta_{\lambda}}{d\lambda}=0,\\ \frac{d h_{\lambda}}{d\lambda}=\begin{pmatrix}
      \cos(\beta(z)\lambda)&-\sin(\beta(z)\lambda)\\ \sin(\beta(z)\lambda)&\cos(\beta(z)\lambda)
  \end{pmatrix}\cdot \begin{pmatrix}
      -y_{\lambda}\beta'(z)\\x_{\lambda}\beta'(z)
  \end{pmatrix} $}
  
 \footnotesize   \[  \Phi_{\lambda}\begin{pmatrix}
      x\\y\\z
  \end{pmatrix}=\begin{pmatrix} \cos(\beta(z)\lambda)\cdot x + \sin(\beta(z)\lambda)\cdot y\\ -\sin(\beta(z)\lambda)\cdot x + \cos(\beta(z)\lambda)\cdot y\\ z\end{pmatrix}; \ \ \ \   B_{\lambda}(x,y,z)=\begin{pmatrix}
      \cos(\beta(z)\lambda) & \sin(\beta(z)\lambda)\\-\sin(\beta(z)\lambda) & \cos(\beta(z)\lambda)
  \end{pmatrix};\]\[ \eta_{\lambda}(z)=1;\ \ \ \ \ \ \ \   h_{\lambda}(x,y,z)=\begin{pmatrix}
      -\beta'(z)\lambda \cdot y \\ \beta'(z)\lambda\cdot x
  \end{pmatrix}.\] 
  \normalsize Since $ V_{\beta}$ is a symmetry, by construction we have that $ (P_{T_{\lambda}}(X,W))$ solves the same  SDE as $(X,W)$, that is
\footnotesize  \[ \underbrace{d \begin{pmatrix}
      \cos(\beta\lambda)\cdot X_t + \sin(\beta\lambda)\cdot Y_t\\
      -\sin(\beta\lambda)\cdot X_t + \cos(\beta\lambda)\cdot Y_t \\ Z_t
  \end{pmatrix}}_{dP_{T_{\lambda}}(X)=d \tilde{W}_t}=\underbrace{\begin{pmatrix}
      \cos(\beta\lambda)&\sin(\beta\lambda)\\-\sin(\beta\lambda)&\cos(\beta\lambda)\\0&0
  \end{pmatrix}\begin{pmatrix}
      dW^1_t\\dW^2_t
  \end{pmatrix}-\begin{pmatrix}
      \beta'\lambda \sin(\beta\lambda)\cdot X_t-\beta'\lambda \cos(\beta\lambda)\cdot Y_t\\
       \beta'\lambda \cos(\beta\lambda)\cdot X_t+\beta'\lambda \sin(\beta\lambda)\cdot Y_t
  \end{pmatrix}dt}_{dP_{T_{\lambda}}(W)= dW'_t-hdt} \]
  \normalsize In particular, $ (X,W)$ under $\mathbb{P}$ has the same law as $ (P_{T_{\lambda}}(X,W))$ under $\mathbb{Q}_{\lambda}$: $\ P_{T_{\lambda}}(X)=\tilde{W}$ is a $\mathbb{Q}_{\lambda}-$Brownian motion.\\
In the same way, if we consider the deterministic time change $ t'=f(t)$, then $ W'_{t'}=\int_0^{t'} \sqrt{f'(f^{-1}(s))}dW_{f^{-1}(s)}$ is still a $\mathbb{P}$-Brownian motion by Proposition \ref{invarianza per rotazioni}.

The symmetry $V_{\alpha}$ provides a generalization of the previous fact, considering $\tilde{W}_{t'}= \sqrt{f'(f^{-1}(t'))}W_{f^{-1}(t')}$ instead of 
the integral $\int_0^{t'} \sqrt{f'(f^{-1}(s))}dW_{f^{-1}(s)}$. Again, since $ W'_{t'}$ is already a $\mathbb{P}$-Brownian motion, $\tilde{W'}_{t'}$ cannot 
be himself a $\mathbb{P}$-Brownian motion. But, since $V_{\alpha}$ is a symmetry, using Girsanov transformation we have that $\tilde{W'}$ is a $\mathbb{Q}$-Brownian motion, where $\mathbb{Q}$ has density with respect to $\mathbb{P}$ given by \small $ \frac{d\mathbb{Q}}{d\mathbb{P}}|_{\mathcal{F}_T}=\exp\Bigg( \int_0^T -\frac{f''(s)}{2\sqrt{f'(s)}}W_s dW_s + \frac{1}{2}\int_0^T \Big|\frac{f''(s)}{2\sqrt{f'(s)}}W_s\Big|^2 ds\Bigg).$ \normalsize 
This first example of two-dimensional Brownian motion shows how the theory of symmetries applied to SDEs is a useful tool for studying invariance properties of stochastic processes and analyzing their laws.
\end{osservazione}
\noindent We now want to apply Theorem \ref{teo integrazione per parti 2} to the family $V_{\beta}$. First, we need to verify Hypothesis A' \eqref{ipotesi A'}:
\begin{equation*}
   \begin{matrix}
    H_{\alpha}=\begin{pmatrix}
       -y\beta'(z)\\x\beta'(z)
   \end{pmatrix}; \ \ \ \ \ \ \ \ \ \ \ \ \
        C_{\alpha,k} H_{k} =\begin{pmatrix}
            \beta(z)\beta'(z)x\\\beta(z)\beta'(z)y
        \end{pmatrix};
        \\ Y(H_{\alpha})= \beta(z) y \partial_x \begin{pmatrix}
         -y\beta'(z)\\
            x\beta'(z)
    \end{pmatrix} - \beta(z) x \partial_y \begin{pmatrix}
 -y\beta'(z)\\
            x\beta'(z)
    \end{pmatrix}=
    \begin{pmatrix}
        \beta(z) \beta'(z) x \\
        \beta(z) \beta'(z) y
    \end{pmatrix};\\
    L(Y^i) =
    [\partial_z + \frac{1}{2}\partial^2_x + \frac{1}{2}\partial^2_y]\begin{pmatrix}
        \beta(z)y\\-\beta(z)x\\0
    \end{pmatrix}=
    \begin{pmatrix}
        \beta'(z) \\ -\beta'(z) \\ 0 
    \end{pmatrix};\\
     \Sigma_{\alpha}(Y^i) = \begin{bmatrix}
        \Sigma_{\alpha}(Y^1)|&\Sigma_{\alpha}(Y^2)|&\Sigma_{\alpha}(Y^3)
    \end{bmatrix}=\\
    =\begin{bmatrix}
      [   \sigma^j_{\alpha}\partial_j ](\beta(z)y)&      [   \sigma^j_{\alpha}\partial_j ](-\beta(z)x)&      [   \sigma^j_{\alpha}\partial_j ](0)
    \end{bmatrix}    =\begin{bmatrix}
        \begin{pmatrix}
            0\\\beta(z)
        \end{pmatrix}| & \begin{pmatrix}
            -\beta(z) \\0
        \end{pmatrix}| \begin{pmatrix}
            0\\0
        \end{pmatrix}
    \end{bmatrix};
    \\
    L(Y(Y^i))=[\partial_z + \frac{1}{2}\partial^2_x + \frac{1}{2}\partial^2_y]\Bigg( \beta(z)y\partial_x\begin{pmatrix}
        \beta(z)y\\-\beta(z)x\\0
    \end{pmatrix} - \beta(z)x\partial_y\begin{pmatrix}
        \beta(z)y\\-\beta(z)x\\0
    \end{pmatrix}  \Bigg)=\\
   = [\partial_z + \frac{1}{2}\partial^2_x + \frac{1}{2}\partial^2_y]\begin{pmatrix}
        -\beta(z)^2x \\ -\beta(z)^2 y \\ 0
        \end{pmatrix}= \begin{pmatrix}
            -2\beta(z)\beta'(z) \\ -2\beta(z)\beta'(z) y \\ 0
        \end{pmatrix} ;      \\   
        \Sigma_{\alpha}(Y(Y^i))=\begin{bmatrix}
           \Sigma_{\alpha}(Y(Y^1))|&\Sigma_{\alpha}(Y(Y^2))|&\Sigma_{\alpha}(Y(Y^3))
        \end{bmatrix}=\\=\begin{bmatrix}
             [  \sigma^j_{\alpha}\partial_j ](-\beta(z)^2x)& [  \sigma^j_{\alpha}\partial_j ](-\beta(z)^2y)& [   \sigma^j_{\alpha}\partial_j ](0)
        \end{bmatrix}=\begin{bmatrix}
            \begin{pmatrix}
                -\beta(z)^2\\0
            \end{pmatrix}|&\begin{pmatrix}
                0\\-\beta(z)^2
            \end{pmatrix}|&\begin{pmatrix}
                0\\0
            \end{pmatrix}
        \end{bmatrix}.
    \end{matrix}
\end{equation*}
Since the previous expressions are continuous in  $z$ at most with linear growth in $ x$ and $y$, we have that $C_{\alpha,k} H_{k}, Y(H_{\alpha}), L(Y^i), \Sigma_{\alpha}(Y^i), L(Y(Y^i)), \Sigma_{\alpha}(Y(Y^i)) \in L^2$  since are polynomials of Gaussian r.v. \noindent
Applying Theorem \ref{teo integrazione per parti 2} we finally get our integration by parts formula for any arbitrary function of time $\beta$: 
\begin{equation}\label{ibp vb}
    0 = \mathbb{E}_{\mathbb{P}} \Big[ F(X_t,Y_t) \int_0^t - Y_s \beta'(z)dW^1_s + X_s \beta'(z) d W^2_s  \Big] + \mathbb{E}_{\mathbb{P}}\Big[\beta(t) Y_t\partial_x F(X_t,Y_t)- \beta(t)X_t \partial_y F(X_t,Y_t) \Big].
\end{equation}
\\
In the same way, we can apply Theorem \ref{teo integrazione per parti 2} to the family $V_{\alpha}$.  First of all, we need to verify, once again, Hypothesis A' \eqref{ipotesi A'}:
\begin{equation*}
    H_{\alpha}=\begin{pmatrix}
        -\frac{1}{2}x\alpha'(z)\\ \frac{1}{2}y\alpha'(z)
    \end{pmatrix};\ \ \ \ \ \ 
    C_{\alpha,k }H_{k}=\begin{pmatrix}
        0 \\0
    \end{pmatrix};
\end{equation*}
\begin{equation*}
    Y(H_{\alpha})=[\frac{1}{2}\alpha(z)x\partial_x + \frac{1}{2}\alpha(z)y\partial_y+\int \alpha(z)dz \partial_z]\begin{pmatrix}
        -\frac{1}{2}x\alpha'(z)\\ \frac{1}{2}y\alpha'(z)
    \end{pmatrix}=\begin{pmatrix}
        -\frac{1}{4}\alpha'(z)\alpha(z)x -\frac{1}{2}x \alpha''(z)\int\alpha(z)dz\\
        \frac{1}{4}\alpha'(z)\alpha(z)x +\frac{1}{2}y \alpha''(z)\int\alpha(z)dz
    \end{pmatrix};
\end{equation*}
\begin{equation*}
    L(Y^i)=[\partial_z + \frac{1}{2}\partial^2_x + \frac{1}{2}\partial^2_y]\begin{pmatrix}
        \frac{1}{2}\alpha(z)x \\ \frac{1}{2}\alpha(z)y\\ \int \alpha(z)dz
    \end{pmatrix}=\begin{pmatrix}
        \frac{1}{2}\alpha'(z)x \\ \frac{1}{2}\alpha'(z) y \\ \alpha(z)
    \end{pmatrix};
\end{equation*}
\begin{equation*}
    \Sigma_{\alpha}(Y^i)= \begin{bmatrix}
        \Sigma_{\alpha}(Y^1)|&\Sigma_{\alpha}(Y^2)|&\Sigma_{\alpha}(Y^3)
    \end{bmatrix}   =\begin{bmatrix}
      [\sigma^j_{\alpha}\partial_j](\frac{1}{2}\alpha(z)x)&[\sigma^j_{\alpha}\partial_j](\frac{1}{2}\alpha(z)y)&[\sigma^j_{\alpha}\partial_j](\int \alpha(z)dz)
    \end{bmatrix} =\end{equation*}\begin{equation*}=\begin{bmatrix}
     \begin{pmatrix}   \frac{1}{2}\alpha(z)\\0 \end{pmatrix}|&\begin{pmatrix}
         0\\ \frac{1}{2}\alpha(z)
     \end{pmatrix}|&\begin{pmatrix}
         0\\0
     \end{pmatrix}
    \end{bmatrix} ;
\end{equation*}
\begin{equation*}
    L(Y(Y^i))=\Big[\partial_z + \frac{1}{2}\partial^2_x + \frac{1}{2}\partial^2_y\Big]\Bigg(\Big[\frac{1}{2}\alpha(z)x\partial_x + \frac{1}{2}\alpha(z)y\partial_y+\int \alpha(z) dz \partial_z\Big] \begin{pmatrix}
        \frac{1}{2}\alpha(z) x \\ \frac{1}{2}\alpha(z) y \\ \int \alpha(z) dz
    \end{pmatrix}\Bigg) =
\end{equation*}
\begin{equation*}
    = [\partial_z + \frac{1}{2}\partial^2_x + \frac{1}{2}\partial^2_y]\begin{pmatrix}
        \frac{1}{4}\alpha^2(z) x + \frac{1}{2}  \alpha'(z) x \int \alpha(z) dz \\ \frac{1}{4}\alpha^2(z) y + \frac{1}{2} \alpha'(z) y \int \alpha(z) dz \\ \alpha(z) \int \alpha(z) dz
    \end{pmatrix}= \begin{pmatrix}
        \alpha(z) \alpha'(z) x + \frac{1}{2} x \alpha''(z) \int \alpha(z) dz \\  \alpha(z) \alpha'(z) y + \frac{1}{2} y \alpha''(z) \int \alpha(z) dz  \\ \alpha'(z) \int \alpha(z) dz + \alpha^2(z)
    \end{pmatrix};
\end{equation*}
\begin{equation*}
    \Sigma_{\alpha}(Y(Y^i))=\begin{bmatrix}
           \Sigma_{\alpha}(Y(Y^1))|&\Sigma_{\alpha}(Y(Y^2))|&\Sigma_{\alpha}(Y(Y^3))
        \end{bmatrix}=
\end{equation*}
\begin{equation*}
    =\begin{bmatrix}
        \begin{pmatrix}
            \frac{1}{4}\alpha^2(z)+\frac{1}{2}\alpha'(z)\int \alpha(z)dz \\ 0
        \end{pmatrix}|\begin{pmatrix}
            0\\ \frac{1}{4}\alpha^2(z)+\frac{1}{2}\alpha'(z)\int \alpha(z)dz 
        \end{pmatrix}| \begin{pmatrix}
            0\\0
        \end{pmatrix}
    \end{bmatrix}.
\end{equation*}
Even in this case the previous expressions are continuous in  $z$ at most with linear growth in $x$ and $y$. Thus, we have that $H_{\alpha},C_{\alpha,k} H_{k}, Y(H_{\alpha}), L(Y^i), \Sigma_{\alpha}(Y^i), L(Y(Y^i)), \Sigma_{\alpha}(Y(Y^i)) \in L^2$  since are polynomials of Gaussian r.v. \noindent
Applying Theorem \ref{teo integrazione per parti 2} we finally get our integration by parts formula for $V_{\alpha}$ for any arbitrary function of time $\alpha$: 
\begin{equation*}
    \int_0^t \alpha(z) dz \cdot \mathbb{E}_{\mathbb{P}}\Big[\frac{1}{2} D^2F(X_t,Y_t)\Big]= \mathbb{E}_{\mathbb{P}}\Big[F(X_t,Y_t) \int_0^t -\frac{1}{2}X_t \alpha'(z) dW^1_t + \frac{1}{2} Y_t \alpha'(z) dW^2_t\Big] +
\end{equation*}
\begin{equation}\label{ibp va}
    + \mathbb{E}_{\mathbb{P}}\Big[\frac{1}{2}\alpha(t) X_t \partial_x F(X_t,Y_t) + \frac{1}{2} \alpha(t) Y_t \partial_y F(X_t,Y_t)\Big].
\end{equation}
\begin{osservazione}\label{remark stein}
With appropriate choices of $\alpha$ and $\beta$,  it is possible to recover well-known formulas from probability theory. For instance, choosing $\alpha$ to be a (non zero) constant, then, dividing  (\ref{ibp va}) by $\alpha$ we get
    \begin{equation}\label{stein va}t \cdot \mathbb{E}_{\mathbb{P}}\Big[ \partial_x^2 F(X_t,Y_t)+\partial_y^2F(X_t,Y_t)\Big]=\mathbb{E}_{\mathbb{P}}\Big[X_t\partial_x F(X_t,Y_t)+Y_t\partial_y F(X_t,Y_t)\Big],\end{equation}
    that is exactly what Stein's lemma states for $ (X_t,Y_t)$ solving  \eqref{BM2}. Indeed, Stein's lemma says that random vector  $  X = (X_1, X_2,.., X_d)$ is  centered Gaussian if and only if \begin{equation}\label{stein}\mathbb{E}_{\mathbb{P}}\Bigg[\sum_{i=1}^d X_i \partial_i f(X)\Bigg]=\mathbb{E}_{\mathbb{P}}\Bigg[\sum_{i,j=1}^d Cov(X_i,X_j) \partial_{ij} f(X)\Bigg]\end{equation} for all smooth and bounded enough $ f: \mathbb{R}^d \rightarrow \mathbb{R}.$ In our example, $(X,Y)$ is a two dimensional Brownian motion, and so $ X_t, Y_t$ are independent and $\sim \mathcal{N}(0,t)$. In particular, $ Cov(X_t,X_t)=t$ and $Cov(X_t,Y_t)=0$, hence (\ref{stein}) and (\ref{stein va}) coincide.\\
    Similarly, choosing $\beta$ (non-zero) constant, then, dividing by $\beta$ (\ref{ibp vb}) becomes
    \begin{equation}\label{isserlis vb}
        \mathbb{E}_{\mathbb{P}}\Big[Y_t \partial_xF(X_t,Y_t)\Big]=\mathbb{E}_{\mathbb{P}}\Big[X_t \partial_y F(X_t,Y_t)\Big]
    \end{equation}
   which is a consequence of Isserli-Wick's theorem in probability theory. Indeed, Isserli-Wick's theorem states that if $ X=(X_1,...,X_d)$ is a zero-mean multivariate normal random vector, then $\forall j=1,..,d$
\begin{equation}\label{isserli}
    \mathbb{E}_{\mathbb{P}}\Big[X_j f(X_1,...,X_d)\Big]=\sum_{i=1}^d Cov(X_j,X_i)\mathbb{E}_{\mathbb{P}}\Big[\partial_i f(X_1,..,X_d)\Big].
\end{equation}
Applying Isserli's theorem to $ f=\partial_x F$ and $ f=\partial_y F$ respectively, then, thanks to Schwarz's theorem, both the right hand side and the left hand side of (\ref{isserlis vb}) are equal to  $ t\cdot \mathbb{E}_{\mathbb{P}}[\partial_{xy} F(X_t,Y_t)]$, hence they coincide and (\ref{isserlis vb}) is recovered. 
    
\end{osservazione}

\subsection*{An additive perturbation of Brownian motion}
Let us now consider the following SDE
\begin{equation}\label{additive mb}
    \begin{pmatrix}
        dX_t\\dY_t\\dZ_t
    \end{pmatrix}= a(X^2_t + Y^2_t) \begin{pmatrix}
        X_t\\Y_t\\0
    \end{pmatrix}dt + b(X^2_t + Y^2_t) \begin{pmatrix}
        -Y_t\\X_t\\0
    \end{pmatrix}dt + \begin{pmatrix}
        0\\0\\1
    \end{pmatrix}dt
    + \begin{pmatrix}
        dW^1_t\\dW^2_t\\0
    \end{pmatrix},
\end{equation}
where $a, b$ have at most polynomial growth and $ a$ is a negative function.
Naturally, \eqref{additive mb} can be restated as follows to emphasize the forms of  $\mu$ and $\sigma$:
\begin{equation*}
    \begin{pmatrix}
        dX_t\\dY_t\\dZ_t
    \end{pmatrix}=\underbrace{\begin{pmatrix}
        a(R^2_t)X_t-b(R^2_t)Y_t\\a(R^2_t)Y_t+b(R^2_t)X_t\\1
    \end{pmatrix}}_{\mu}dt+\underbrace{\begin{pmatrix}
        1 & 0\\0 & 1 \\ 0 & 0
    \end{pmatrix}}_{\sigma}\begin{pmatrix}
        dW^1_t \\ dW^2_t
    \end{pmatrix}
\end{equation*}
where $R^2_t=X^2_t+Y^2_t$. To compute the symmetries of (\ref{additive mb}) we consider the determining equation (\ref{equazioni determinanti}) and we look for
 infinitesimal symmetries $V=(Y,C,\tau,H)$ with
\begin{equation*}
\begin{pmatrix}
   Y= \begin{pmatrix}
        f(x,y,z)\\
        g(x,y,z)\\
        m(x,y,z)
    \end{pmatrix};
 &

    C= \begin{pmatrix}
        0 & c(x,y,z) \\
        -c(x,y,z) & 0
    \end{pmatrix};
    &   \tau=\tau(x,y,z);

    &   H=\begin{pmatrix}
        H_1(x,y,z)\\
        H_2(x,y,z)
    \end{pmatrix}
\end{pmatrix}
\end{equation*} 
satisfying the determining equations. From the second determining equation in \eqref{equazione determinante 1.1}, after straightforward computations, we obtain

\begin{equation}\label{additive mb 1}
    \begin{matrix}
        f_x(x,y,z)=\frac{1} {2}\tau(x,y,z),&f_y(x,y,z)=c(x,y,z),\\
        g_x(x,y,z)=-c(x,y,z),&
        g_y(x,y,z)=\frac{1}{2}\tau(x,y,z),\\
        m_x(x,y,z)=0&m_y(x,y,z)=0.
            \end{matrix}
\end{equation}
On the other hand, from the first determining equation in \eqref{equazione determinante 1.1} we obtain
\begin{equation*}
2x^2 f a'(r^2)+f a(r^2)-2xy f b'(r^2)+2xy g a'(r^2)-2y^2g b'(r^2) -g b(r^2)   -\frac{1}{2}f_{xx}-\frac{1}{2}f_{yy}+
\end{equation*}
\begin{equation}\label{additive mb 2}
    -(a(r^2)x-b(r^2)y)f_x- (a(r^2)y+b(r^2)x)f_y-f_z-H_1+\tau(    a(r^2)x-b(r^2)y)=0 
\end{equation}
\\
\begin{equation*}
       2xy f a'(r^2)+f b(r^2)+2x^2f b'(r^2)+g a(r^2)+2y^2g a'(r^2)+2xy g b'(r^2) -\frac{1}{2}g_{xx}-\frac{1}{2}g_{yy}+
\end{equation*}
\begin{equation}\label{additive mb 3}
-(a(r^2)x-b(r^2)y)g_x  - (a(r^2)y+b(r^2)x)g_y-g_z-H_2+\tau(a(r^2)y+b(r^2)x)    =0
\end{equation}
\\
\begin{equation}\label{additive mb 4}
     -\frac{1}{2}m_{xx}-\frac{1}{2}m_{yy}-(a(r^2)x-b(r^2)y)m_x - (a(r^2)y+b(r^2)x)m_y-m_z+\tau=0.
\end{equation}
\\
As in the previous case, from $ m_x=m_y=0$ two choices are possible: either $ m\equiv 0$ or $ m=m(z)$. \\
Hence, the first family of symmetries is obtained by choosing $m \equiv 0$. From \eqref{additive mb 4} we thus have that also $\tau=0$,
so, from  \eqref{additive mb 1} we obtain 
\begin{equation*}
    c(x,y,z)=\beta(z),\ \ \ \ \ \ f(y,z)=\beta(z)y, \ \ \ \ \ \  g(x,z)=-\beta(z)x
\end{equation*}
with $\beta$ an arbitrary function of time. Replacing these quantities into \eqref{additive mb 2} and \eqref{additive mb 3} we obtain
\begin{equation*}
    H_1=-y\beta'(z), \ \ \ \  \ \ \ \ \ \ H_2 =x\beta'(z).
\end{equation*}
Hence, we get the first family of symmetry of (\ref{additive mb}) 
\begin{equation}\label{simmetria additive mb 1}
    V_{\beta}=\Bigg( \begin{pmatrix}
        \beta(z)y\\-\beta(z)x\\0
    \end{pmatrix}, \begin{pmatrix}
        0 & \beta(z) \\ - \beta(z) & 0 
    \end{pmatrix}, 0, \begin{pmatrix}
        -y\beta'(z)\\ x\beta'(z)
    \end{pmatrix}\Bigg)
\end{equation}
depending on an arbitrary function $\beta(z)$.\\
\\
The second family is obtained by choosing $ m=m(z).$ From \eqref{additive mb 4} we thus have $\tau=m_z(z)$, so we can put
\begin{equation*}
 \tau=\alpha(z), \ \ \ \ \ \ \ \ \ \ \ m(z)=\int \alpha(z)dz,
\end{equation*}

with $\alpha$ an arbitrary function of time. From \eqref{additive mb 1} we obtain
\begin{equation*}
    f=\frac{1}{2}\alpha(z)x, \ \ \ \ \ \ g=\frac{1}{2}\alpha(z)y, \ \ \ \ \ \ \ c=0.
\end{equation*}
Replacing these equalities into \eqref{additive mb 2} and \eqref{additive mb 3} we obtain
\begin{equation*}H_1= \alpha(z) x^3a'(r^2)-x^2y\alpha(z)b'(r^2)+xy^2\alpha(z)a'(r^2)-y^3\alpha(z)b'(r^2)-\frac{1}{2}\alpha'(z)x +x\alpha(z)a(r^2)-y\alpha(z)b(r^2)
\end{equation*} 
\begin{equation*}
    H_2=  -\frac{1}{2}y\alpha'(z) +\alpha(z) y^3 a'(r^2)+\alpha(z) x^2ya'(r^2)+ \alpha(z) x y^2 b'(r^2)+\alpha(z) x^3 b'(r^2) +\alpha(z) y a(r^2)  + \alpha(z) x b(r^2) .
\end{equation*}

Thus, the second family of symmetries of \eqref{additive mb} is the following 
\begin{equation}\label{simmetria additive mb 2}
    V_{\alpha}=\Bigg(  \begin{pmatrix}
    \frac{1}{2}\alpha x \\ \frac{1}{2}\alpha y \\ \int \alpha dz
\end{pmatrix}, \begin{pmatrix}
    0 & 0\\ 0 & 0
\end{pmatrix}, \alpha, \begin{pmatrix}
    \alpha x^3a'-x^2y\alpha b'+xy^2\alpha a'
 -y^3\alpha b'-\frac{1}{2}\alpha'x +x\alpha a-y\alpha b
    \\
    -\frac{1}{2}y\alpha'+\alpha y^3 a'+\alpha x^2ya'+\alpha x y^2 b'+\alpha x^3 b' +\alpha y a + \alpha x b
\end{pmatrix}\Bigg).
\end{equation}
\\

We now want to apply Theorem \ref{teo integrazione per parti 2} to $ V_{\alpha}$ (\ref{simmetria additive mb 1}) and $ V_{\beta}$ (\ref{simmetria additive mb 2}).\\
Let us begin with $\ V_{\beta}$. 
In order to prove that Hypothesis A' holds, we have to check that
\begin{equation*}
    \begin{matrix}
    H_{\alpha}=\begin{pmatrix}
        -y\beta'(z)\\x\beta'(z)
    \end{pmatrix} \; \; 
        C_{\alpha,k} H_{k}
    =\begin{pmatrix}
        \beta(z)\beta'(z)x\\\beta(z)\beta'(z)y
    \end{pmatrix}
            \\ Y(H_{\alpha})= 
    \begin{pmatrix}
        \beta(z) \beta'(z) x \\
        \beta(z) \beta'(z) y
    \end{pmatrix}\\
    L(Y^i) =
    \begin{pmatrix}
        \beta'(z)\big( 1 - (a(r^2)y+b(r^2)x\big) \\ -\beta'(z)\big(1-(a(r^2)x-b(r^2)y\big) \\ 0 
    \end{pmatrix}\\
    \Sigma_{\alpha}(Y^i) =
   \begin{bmatrix}
            \begin{pmatrix}
                0\\ \beta(z)
            \end{pmatrix}| \begin{pmatrix}
                -\beta(z)\\0
            \end{pmatrix}|\begin{pmatrix}
                0\\0
            \end{pmatrix}
        \end{bmatrix}\\    
    L(Y(Y^i)) 
        = \begin{pmatrix}
            -2\beta(z)\beta'(z) -\beta^2(z)(a(r^2)x-b(r^2)y)\\ -2\beta(z)\beta'(z) y (a(r^2)y + b(r^2)x) \\ 0
        \end{pmatrix}\\
        \Sigma_{\alpha}(Y(Y^i))
        =\begin{bmatrix}
            \begin{pmatrix}
                -\beta(z)^2\\0
            \end{pmatrix}|&\begin{pmatrix}
                0\\-\beta(z)^2
            \end{pmatrix}|&\begin{pmatrix}
                0\\0
            \end{pmatrix}
        \end{bmatrix}
         \end{matrix}
\end{equation*}
$\in L^2(\mathbb{P})$. In the previous example  we easily concluded by saying that all this terms were polynomials of Gaussian random variables $ X$ and $Y$. Now, $ X$ and $ Y$ are no more Gaussian under the probability $\mathbb{P}$ (we added a drift and the solution is no more the $\mathbb{P}-$ Brownian motion), so we must use the Lyapunov approach.
Hypothesis A' applied to this example requires the integrability of terms given by the product of powers of $ X_t,Y_t$ and the functions $a(R^2_t),b(R^2_t).$ Consider the function \[ \phi(t,x,y)=e^{k(t)(x^2+y^2)}=e^{k(t)r^2}.\]
If $ k(t)\geq0$, since both $a$ and $b$ have at most polynomial growth, $\phi$ provides an upper bound for all the required quantities. Furthermore, since in this example we take  $ a$ negative we have that
\begin{equation*}
    L(\phi)=\Big[\Big(k'(t)+2k^2(t)+2k(t)a(r^2)\Big)r^2+2k(t)\Big]\phi \leq \Big[\Big(k'(t)+2k^2(t)\Big)r^2+2k(t)\Big]\phi .
\end{equation*}
Thus, if we take $ k$ satisfying the equation \[ k'+2k^2<0 \]  we have that $ L(\phi)\leq 2k(t) \phi$, so $\phi$ is a Lyapunov function and the integrability of the terms required is guaranteed.  Hence, we can apply Theorem \ref{teo integrazione per parti 2} to the family $ V_{\beta}$, obtaining 
\begin{equation*}
    0 = \mathbb{E}_{\mathbb{P}} \Big[ F(X_t,Y_t) \int_0^t - Y_s \beta'(z)dW^1_s + X_s \beta'(z) d W^2_s  \Big] + \mathbb{E}_{\mathbb{P}}\Big[\beta(t) Y_t\partial_x F(X_t,Y_t)  - \beta(t)X_t \partial_y F(X_t,Y_t) \Big].
\end{equation*}

Concerning $ V_{\alpha}$, to verify Hypothesis A' (\ref{ipotesi A'}) we have to require that
\begin{equation*}
    H_{\alpha}= \begin{pmatrix}
    \alpha x^3a'-x^2y\alpha b'+xy^2\alpha a'+
 -y^3\alpha b'-\frac{1}{2}\alpha'x +x\alpha a-y\alpha b
    \\
    -\frac{1}{2}y\alpha'+\alpha y^3 a'+\alpha x^2ya'+\alpha x y^2 b'+\alpha x^3 b' +\alpha y a  + \alpha x b
\end{pmatrix}
\end{equation*}
\begin{equation*}
C_{\alpha k}H_k=\begin{pmatrix}
    0\\0
\end{pmatrix} \ \ \ \ \ \ \ \ \ \ \ \ \ \ \     Y(H_{\alpha})=\begin{pmatrix}
       Y(H_1)\\
       Y(H_2)
        \end{pmatrix} 
\end{equation*}
where
\begin{equation*}
    Y(H_1)=-\frac{\alpha\alpha'+2\alpha''\int \alpha(z)dz}{4}x+a(r^2)\Big(\frac{\alpha^2}{2}x+\alpha'\int \alpha(z)dz\Big)+
    \end{equation*}
\begin{equation*}
    +a'(r^2)\Big(\frac{5\alpha^2+2\alpha'\int\alpha(z)dz}{2}x^3+\frac{3\alpha^2}{2}xy^2+\alpha^2xy\Big)+\alpha^2 a''(r^2)\Big(x^5+x^3y^2+xy^4+x^2y\Big)+
\end{equation*}
\begin{equation*}
    -b(r^2)(\frac{\alpha^2}{2}y+\alpha'\int\alpha(z)dz\Big)-b'(r^2)\Big(\frac{a^2}{2}x^2+\frac{\alpha^2}{2}y^4+\alpha^2y^3+
\end{equation*}
\begin{equation*}
    +(2\alpha^2+\alpha'\int\alpha(z)dz)x^2y+\alpha'\int\alpha(z)dz\Big)-\alpha^2b''(r^2)\Big(y^5+2x^2y^3+x^4\Big);
\end{equation*}
\\
\begin{equation*}
    Y(H_2)=-\frac{\alpha\alpha'+2\alpha''\int \alpha(z)dz}{4}y+b(r^2)\Big(\frac{\alpha^2}{2}x+\alpha'\int \alpha(z)dz\Big)+
    \end{equation*}
\begin{equation*}
    b'(r^2)\Big(\frac{5\alpha^2+2\alpha'\int\alpha(z)dz}{2}x^3+\frac{3\alpha^2}{2}xy^2+\alpha^2xy\Big)+\alpha^2 b''(r^2)\Big(x^5+x^3y^2+xy^4+x^2y\Big)+
\end{equation*}
\begin{equation*}
    +a(r^2)(\frac{\alpha^2}{2}y+\alpha'\int\alpha(z)dz\Big)+a'(r^2)\Big(\frac{a^2}{2}x^2+\frac{\alpha^2}{2}y^4+\alpha^2y^3+
\end{equation*}
\begin{equation*}
    +(2\alpha^2+\alpha'\int\alpha(z)dz)x^2y+\alpha'\int\alpha(z)dz\Big)+\alpha^2a''(r^2)\Big(y^5+2x^2y^3+x^4\Big);
\end{equation*}
\\
\begin{equation*}
    L(Y^i)=\begin{pmatrix}
        \frac{1}{2}\alpha'(z)x + \frac{1}{2}\alpha(z)(  a(r^2)x - b(r^2)y )\\ \frac{1}{2}\alpha'(z) y + \frac{1}{2} \alpha(z) + (a(r^2)y + b(r^2)x) \\ \alpha(z)
    \end{pmatrix} 
\end{equation*}
\begin{equation*}
    L(Y(Y^i))=\begin{pmatrix}
        \alpha(z) \alpha'(z) x + \frac{1}{2} x \alpha''(z) \int \alpha(z) dz + (\frac{1}{4}\alpha^2(z)+\frac{1}{2}\alpha'(z)\int \alpha(z) dz ) (  a(r^2)x - b(r^2)y )    \\
        \alpha(z) \alpha'(z) y + \frac{1}{2} y \alpha''(z) \int \alpha(z) dz  +  (\frac{1}{4}\alpha^2(z)+\frac{1}{2}\alpha'(z)\int \alpha(z) dz ) (  a(r^2)y+ b(r^2)x ) 
        \\ \alpha'(z) \int \alpha(z) dz + \alpha^2(z)
    \end{pmatrix}
\end{equation*}
\begin{equation*}
    \Sigma_{\alpha}(Y^i)=  \begin{bmatrix}
        \begin{pmatrix}
            \frac{1}{2}\alpha(z) \\ 0
        \end{pmatrix}|&\begin{pmatrix}
            0\\\frac{1}{2}\alpha(z)
        \end{pmatrix}|&\begin{pmatrix}
            0\\0
        \end{pmatrix}
    \end{bmatrix}
\end{equation*}
\begin{equation*}
    \Sigma_{\alpha}(Y(Y^i))=\begin{bmatrix}
        \begin{pmatrix}
            \frac{1}{4}\alpha^2(z)+\frac{1}{2}\alpha'(z)\int \alpha(z) dz \\ 0
        \end{pmatrix}|&\begin{pmatrix}
            0\\  \frac{1}{4}\alpha^2(z)+\frac{1}{2}\alpha'(z)\int \alpha(z) dz 
        \end{pmatrix}|&\begin{pmatrix}
            0\\0
        \end{pmatrix}
    \end{bmatrix}.
\end{equation*}
are in $ L^2(\mathbb{P}).$
So, we need to verify the integrability of terms given by the product of powers of $ x,y$, functions $ a,b$ and their first and second derivatives. The Lyapunov function $\phi(t,x,y)=e^{k(t)(x^2+y^2)}$ considered for the family $ V_{\beta}$ still provides an upper bound for all the quantities required by the hypothesis applied to $ V_{\alpha}$. Since we already proved that $\phi$ is Lyapunov for this equation, Hypothesis A' is verified. Hence,  we can apply Theorem \ref{teo integrazione per parti 2} also for $ V_{\alpha}$:
\begin{equation*}
    \int \alpha(z) dz \cdot \mathbb{E}_{\mathbb{P}}[\big( a(r^2)x - b(r^2)y\big)\partial_x F(X_t,Y_t) + \big(a(r^2)y+b(r^2)x \big)\partial_y F(X_t,Y_t) + \frac{1}{2} D^2 F(X_t,Y_t)]=
\end{equation*}
\begin{equation*}
    = \mathbb{E}_{\mathbb{P}}[F(X_t)\int_0^T ( \alpha x^3a'-x^2y\alpha b'+xy^2\alpha a'-y^3\alpha b'-\frac{1}{2}\alpha'x +x\alpha a-y\alpha b)dW^1_t
\end{equation*}
\begin{equation*}
     + (   \frac{1}{2}y\alpha'+\alpha y^3 a'+\alpha x^2ya'+\alpha x y^2 b'+\alpha x^3 b' +\alpha y a  + \alpha x b) dW^2_t]+
\end{equation*}
\begin{equation*}
    +\mathbb{E}_{\mathbb{P}}[\frac{1}{2} \alpha x \partial_x F(X_t,Y_t) + \frac{1}{2}\alpha y \partial_y F(X_t,Y_t)].
\end{equation*}

\subsection*{Stochastic Lotka Volterra model}
Consider now the following SDE:
\begin{equation}\label{Lotka Volterra}
    \begin{pmatrix}
        dX_t\\
        dY_t\\
        dZ_t
    \end{pmatrix}=\begin{pmatrix}
        X_t(\alpha-\beta Y_t)\\Y_t(\delta X_t-\gamma)\\1
    \end{pmatrix}dt + \begin{pmatrix}
        \tilde{\sigma}X_t & 0 \\ 0 & \tilde{\sigma}Y_t \\ 0 & 0
    \end{pmatrix}\begin{pmatrix}
        dW^1_t\\ dW^2_t
    \end{pmatrix}
\end{equation}
which can be seen as the well known predator-prey Lotka Volterra model plus a stochastic noise (see \cite{Arato}).

We are looking for infinitesimal symmetries $ V=(Y,C,\tau,H)$, with
\begin{equation*}
\begin{pmatrix}
   Y= \begin{pmatrix}
        f(x,y,z)\\
        g(x,y,z)\\
        m(x,y,z)
    \end{pmatrix};
 &

    C= \begin{pmatrix}
        0 & c(x,y,z) \\
        -c(x,y,z) & 0
    \end{pmatrix};
    &   \tau=\tau(x,y,z);

    &   H=\begin{pmatrix}
        H_1(x,y,z)\\
        H_2(x,y,z)
    \end{pmatrix}
\end{pmatrix}
\end{equation*} 
satisfying the determining equations from Theorem \ref{equazioni determinanti}, i. e.
\begin{equation*}
    (\alpha-\beta y)f - \beta x g -(\alpha  - \beta y)x f_x - (\delta x- \gamma )yf_y-f_z -\frac{1}{2}\tilde{\sigma}^2x^2f_{xx}-\frac{1}{2}\tilde{\sigma}^2y^2f_{yy}-\tilde{\sigma}xH_1+\tau x(\alpha -\beta y)=0
\end{equation*}
\begin{equation*}
   \delta yf+ (\delta x-\gamma )g  -(\alpha  - \beta y)x g_x - (\delta x- \gamma )yg_y-g_z-\frac{1}{2}\tilde{\sigma}^2x^2g_{xx}-\frac{1}{2}\tilde{\sigma}^2y^2g_{yy}-\tilde{\sigma}yH_2+\tau y(\delta x - \gamma)=0
\end{equation*}
\begin{equation*}
     (\alpha  - \beta y)x m_x + (\delta x- \gamma )ym_y+m_z+\frac{1}{2}\tilde{\sigma}^2x^2m_{xx}+\frac{1}{2}\tilde{\sigma}^2y^2m_{yy}=\tau
\end{equation*}
\begin{equation*}
    f\tilde{\sigma}-\tilde{\sigma}xf_x+\frac{1}
    {2}\tau\tilde{\sigma}x=0
\end{equation*}
\begin{equation*}
    \tilde{\sigma}yf_y=\tilde{\sigma}xc
\end{equation*}
\begin{equation*}
    \tilde{\sigma}x g_x=-\tilde{\sigma }yc
\end{equation*}
\begin{equation*}
    \tilde{\sigma}g-\tilde{\sigma}yg_y+\frac{1}{2}\tau\tilde{\sigma}y=0
\end{equation*}
from which two families of infinitesimal symmetries arise 
\begin{equation*}
    V_a=\Bigg( \begin{pmatrix}
        \frac{1}{2}a(z)x\ln x\\  \frac{1}{2}a(z) y \ln y\\\int a(z)dz
    \end{pmatrix},\begin{pmatrix}
        0&0\\0&0
    \end{pmatrix},a(z),\begin{pmatrix}
        -\frac{1}{4}a(z)\tilde{\sigma}+\frac{1}{2\tilde{\sigma}}a(z)(\alpha-\beta y)-\frac{1}{2\tilde{\sigma}}a'(z) \ln x-\frac{1}{2\tilde{\sigma}}\beta a(z) y\ln y\\
        -\frac{1}{4}a(z)\tilde{\sigma}+\frac{1}{2\tilde{\sigma}}a(z)(\delta x - \gamma)-\frac{1}{2\tilde{\sigma}}a'(z)\ln y+\frac{1}{2\tilde{\sigma}}\delta a(z) x\ln x
    \end{pmatrix}\Bigg);
\end{equation*}
\begin{equation*}
    V_b=\Bigg( \begin{pmatrix}
        b(z)x \ln y\\-b(z)y \ln x\\0
    \end{pmatrix}, \begin{pmatrix}
        0&b(z)\\-b(z)&0
    \end{pmatrix}, 0, \begin{pmatrix}
        -\frac{1}{2}\tilde{\sigma}b(z)+\frac{b(z)}{\tilde{\sigma}}(\delta x - \gamma)-\frac{1}{\tilde{\sigma}}  b'(z)\ln y+\frac{\beta}{\tilde{\sigma}}b(z) y \ln x\\
        \frac{1}{2}\tilde{\sigma}b(z)-\frac{b(z)}{\tilde{\sigma}}(\alpha-\beta y)+\frac{1}{\tilde{\sigma}} b'(z)\ln x +\frac{\delta}{\tilde{\sigma}}b(z) x \ln y
    \end{pmatrix}\Bigg).
\end{equation*}
We now want to apply Theorem \ref{teo integrazione per parti 2}, so we must check Hypothesis A' (\ref{ipotesi A'}).\\
Concerning $ V_a$, we have to verify that
\begin{equation*}
    H=\begin{pmatrix}
        -\frac{1}{4}a(z)\tilde{\sigma}+\frac{1}{2\tilde{\sigma}}a(z)(\alpha-\beta y)-\frac{1}{2\tilde{\sigma}}a'(z) \ln x-\frac{1}{2\tilde{\sigma}}\beta a(z) y \ln y\\
        -\frac{1}{4}a(z)\tilde{\sigma}+\frac{1}{2\tilde{\sigma}}a(z)(\delta x - \gamma)-\frac{1}{2\tilde{\sigma}}a'(z)\ln y+\frac{1}{2\tilde{\sigma}}\delta a(z) x \ln x
    \end{pmatrix};
\end{equation*}
\begin{equation*}
    C_{\alpha,k}H_{k}=\begin{pmatrix}
        0\\0
    \end{pmatrix};
\end{equation*}
\begin{equation*}
    Y(H^1)=- \frac{a(z)a'(z)}{4\tilde{\sigma}} \ln x-\frac{a(z)^2\beta}{2\tilde{\sigma}} y \ln y-\frac{\beta a(z)^2}{4\tilde{\sigma}}y \ln^2 y+
    \end{equation*}
\begin{equation*}
    +\int a(z)dz(-\frac{a'(z)\tilde{\sigma}}{4}+\frac{a'(z)\alpha}{2\tilde{\sigma}}-\frac{a'(z)\beta}{2\tilde{\sigma}}y-\frac{a''(z)}{2\tilde{\sigma}}\ln x-\frac{a'(z)\beta}{2\tilde{\sigma}} y \ln y)
\end{equation*}
\begin{equation*}
    Y(H^2)=    \frac{a(z)^2\delta}{2\tilde{\sigma}} x \ln x+\frac{a(z)^2\delta}{4\tilde{\sigma}}x \ln^2 x-\frac{a(z)a'(z)}{4\tilde{\sigma}} \ln y+
   \end{equation*}
\begin{equation*} 
    +\int a(z)dz(-\frac{a'(z)\tilde{\sigma}}{4}-\frac{a'(z)\gamma}{2\tilde{\sigma}}+\frac{a'(z)\delta}{2\tilde{\sigma}}x-\frac{a''(z)}{2\tilde{\sigma}} \ln y+\frac{a'(z)\delta}{2\tilde{\sigma}} x \ln x)
\end{equation*}
\begin{equation*}
   L(Y^i) = \begin{pmatrix}
        \frac{a(z)\alpha}{2}x-\frac{a(z)\beta}{2}xy+  \frac{a(z)\alpha}{2}x \ln x -\frac{a(z)\beta}{2}xy \ln x+\frac{a'(z)}{2}x \ln x+\frac{\tilde{\sigma}^2a(z)}{4}x \\
        \frac{a(z)\delta}{2}xy-\frac{a(z)\gamma}{2}y+\frac{a(z)\delta}{2}xy \ln y-\frac{a(z)\gamma}{2}y \ln y+\frac{a'(z)}{2}y \ln y+\frac{\tilde{\sigma}^2 a(z)}{4}y\\a(z)
    \end{pmatrix};    
\end{equation*}
\begin{equation*}
\Sigma_{\alpha}(Y^i)=\begin{bmatrix}
      \begin{pmatrix}  \frac{a(z)\tilde{\sigma}}{2}(x+x \ln x)\\0 \end{pmatrix}| & \begin{pmatrix} 0 \\ \frac{a'(z)\tilde{\sigma}}{2}x^2  \ln x\\
        0 \end{pmatrix} |& \begin{pmatrix} 0 \\ 0 \end{pmatrix}
    \end{bmatrix};
    \end{equation*}
\begin{equation*}
  L(Y(Y^1))  = 
        \frac{a(z)^2\alpha}{4}(x+3x \ln x+x \ln^2 x)-\frac{a(z)^2\beta}4(xy+3xy \ln x+ xy \ln^2 x)+\frac{a'(z)\int a(z)dz \alpha}{2}(x+x \ln x)+
 \end{equation*}       
  \begin{equation*}  
        -\frac{a'(z)\int a(z)dz \beta}{2}(xy+xy \ln x)+\frac{a(z) a'(z)}{2}(x \ln x+x \ln^2 x) + \frac{a(z)a'(z)+a''(z)\int a(z)dz}{2}x \ln x +
  \end{equation*}
  \begin{equation*}
    +\frac{a(z)^2\tilde{\sigma}}{8}(3x+2 \ln x)+\frac{a'(z)\int a(z)dz \tilde{\sigma}}{4}x
\end{equation*}

\begin{equation*}
  L(Y(Y^2))  = \frac{a(z)^2\delta}{4}(xy+3xy \ln y+xy \ln^2 y)-\frac{a(z)^2\gamma}{4}(y+3y\ln y+y \ln^2 y)+\frac{a'(z)\int a(z)dz \delta}{2}(xy+xy \ln y) +
  \end{equation*}
  \begin{equation*}
  -\frac{a'(z)\int a(z)dz \gamma}{2}(y+y \ln y)+\frac{a(z) a'(z)}{2}y (\ln y+y  \ln^2 y) + \frac{a(z)a'(z)+a''(z)\int a(z)dz}{2}y \ln y +
  \end{equation*}
  \begin{equation*}
  +\frac{a(z)^2\tilde{\sigma}}{8}(3y+2 \ln y)+\frac{a'(z)\int a(z)dz \tilde{\sigma}}{4}y
\end{equation*}

\begin{equation*}
  L(Y(Y^3))  = a'(z)\int a(z)dz + a^2(z)     
\end{equation*}
\begin{equation*}
    \Sigma_{\alpha}(Y(Y^1))=\begin{pmatrix}
           \frac{\tilde{\sigma}a^2}{4}x(\ln^2x+1+3\ln x)+\frac{\tilde{\sigma}a'\int a dz}{2}x (\ln x+1) \\  0
       \end{pmatrix}
\end{equation*}

\begin{equation*}
    \Sigma_{\alpha}(Y(Y^2))=\begin{pmatrix}
           0\\  \frac{\tilde{\sigma}a^2}{4}y(\ln^2 y+1+3 \ln y)+\frac{\tilde{\sigma}a'\int a dz}{2}y (\ln y +1) 
       \end{pmatrix}
\end{equation*}
\begin{equation*}
    \Sigma_{\alpha}(Y(Y^3))=\begin{pmatrix}
           0\\  0
       \end{pmatrix}
\end{equation*}
are in $ L^2(\mathbb{P}).$ Summarizing, we need to verify the integrability of terms of the form
\begin{equation}\label{*}
    x^2,y^2,x^2y^2,\ln^2 x,\ln^2 y,x^2 \ln^2 x,y^2 \ln^2 y,x^2y^2 \ln^2 x, x^2y^2 \ln^2 y,x^2 \ln^4 x,y^2 \ln^4 y,x^2y^2 \ln^4 x, x^2y^2 \ln^4 y.
\end{equation}
If we take \[\phi(x,y)=Ax^5+By^5+Cx^4y+Dx^3y^2+Ex^2y^3+Fxy^4\]
we have that
\begin{equation*}
    L(\phi)= (5\alpha+10\tilde{\sigma}^2)Ax^5+(4\alpha+6\tilde{\sigma}^2-\gamma)Bx^4y+(3\alpha+3\tilde{\sigma}^2-2\gamma)C x^3y^2+(2\alpha+4\tilde{\sigma}^2-3\gamma)D x^2y^3+
\end{equation*} 
\begin{equation*}
   (\alpha+6\tilde{\sigma}^2-4\gamma)Exy^4+ 
(10\tilde{\sigma}^2-5\gamma)Fy^5+xy\Bigg[ (B\delta-5A\beta)x^4+(2C\delta-4B\beta)x^3y +
\end{equation*}
\begin{equation*}
  +(3C\delta-3B\beta)x^2y^2+(4E\delta-2D\beta)x^3+(5F\delta-E\beta)y^4\Bigg].
  \end{equation*}

So, in order to have $ L(\phi)\leq M\phi$, we  require that
\begin{equation*}
    \begin{matrix}
     (B\delta-5A\beta)\leq 0& (2C\delta-4B\beta)\leq0   &(3D\delta-3C\beta)\leq0&(4E\delta-2D\beta)\leq0&(5F\delta-E\beta)\leq 0
    \end{matrix}
\end{equation*}
 It suffices, for instance, to take $ A=1, B=4\beta/\delta, C=7(\beta/\delta)^2, D=6(\beta/\delta)^3, E=2(\beta/\delta)^4,F=\frac{1}{5}(\beta/\delta)^5.$
  Let us analyze the behavior in the neighborhood of zero by considering
  \[\psi(x,y)=\frac{1}{x^2}+\frac{1}{y^2}.\] As we are considering a neighborhood of $\underline{0}$, and since $ X_t$ and $ Y_t$ are always $\geq0$ (they represent the number of preys and predators at time $ t$), we can restrict to the case $ x,y \in U:=\{B(0,\epsilon \cap \{x>0\}\cap\{y>0\}\}.$ Since in $ U$ $ y \leq 1$, we have that
  \[L(\psi)=(3\tilde{\sigma}^2-2\alpha)\frac{1}{x^2}+(2\gamma+3\tilde{\sigma}^2)\frac{1}{y^2}+2\beta\frac{y}{x^2}-2\delta\frac{x}{y^2}\leq (3\tilde{\sigma}^2-2\alpha)\frac{1}{x^2}+(2\gamma+3\tilde{\sigma}^2)\frac{1}{y^2}+2\beta\frac{1}{x^2}\leq M\psi \]
  for a suitable $ M.$ Hence, for suitable constants $ M,N$ we have that $ (\ref{*})\leq N\phi+M\psi$, with $\phi$ and $\psi$ Lyapunov. Thus, Hypothesis A' is verified and from Theorem \ref{teo integrazione per parti 2} we obtain the following formula for the family $ V_{\alpha}$
  \begin{equation*}
      -\int a(z)dz \mathbb{E}_{\mathbb{P}}[X_t(\alpha-\beta Y_t) F_x(X_t,Y_t)+Y_t(\delta X_t-\gamma)F_y(X_t,Y_t)+\frac{\tilde{\sigma}^2X_t^2}{2}F_{xx}(X_t,Y_t)+\frac{\tilde{\sigma}^2Y_t^2}{2}F_{yy}]+
  \end{equation*}
  \begin{equation*}
      +\mathbb{E}_{\mathbb{P}}\Bigg[F(X_t,Y_t)\int_0^t \Big(-\frac{a(s)\tilde{\sigma}}{4}+\frac{a(s)}{2\tilde{\sigma}}(\alpha-\beta Y_s)-\frac{1}{2\tilde{\sigma}}a'(s)\ln (X_s)-\frac{1}{2\tilde{\sigma}}\beta a(s) Y_s\ln (Y_s)\Big)dW^1_s+
  \end{equation*}
  \begin{equation*}
      \Big(-\frac{a(s)\tilde{\sigma}}{4}+\frac{a(s)}{2\tilde{\sigma}}(\delta X_s-\gamma)-\frac{1}{2\tilde{\sigma}}a'(s)\ln (Y_s)+\frac{1}{2\tilde{\sigma}}\delta a(s) X_s\ln (X_s)\Big)\Bigg]+
  \end{equation*}
  \begin{equation*}
      +\mathbb{E}_{\mathbb{P}}[\frac{a(t)}{2}X_t \ln (X_t) F_x(X_t,Y_t)+\frac{a(t)}{2}Y_t \ln (Y_t)F_y(X_t,Y_t)]+
  \end{equation*}
  \begin{equation*}   
            -\mathbb{E}_{\mathbb{P}}[\frac{a(t)}{2}X_t \ln (X_t) F_x(X_0,Y_0)+\frac{a(t)}{2}Y_t \ln (Y_t)F_y(X_0,Y_0)]=0.
  \end{equation*}
 Concerning the second family of symmetries $V_b$, by using a similar approach to that used in proving the validity of Hypothesis A', it is possible to derive the following integration by parts formula:
 \begin{equation*}
  \mathbb{E}_{\mathbb{P}}\Bigg[ F(X_t,Y_t)\int_0^t \Big( -\frac{\tilde{\sigma}b(s)}{2}+\frac{b(s)}{\tilde{\sigma}}(\delta X_s-\gamma)-\frac{b'(s)}{\tilde{\sigma}}\ln (Y_s)+\frac{\beta b(s)}{\tilde{\sigma}}Y_s \ln (X_s)\Big)dW^1_s+ 
 \end{equation*}
 \begin{equation*}
     +\Big(\frac{\tilde{\sigma}b(s)}{2}-\frac{b(s)}{\tilde{\sigma}}(\alpha-\beta Y_s)+\frac{b'(s)}{\tilde{\sigma}}\ln (X_s)+\frac{\delta b(s)}{\tilde{\sigma}}X_s \ln (Y_s)\Big)dW^2_s\Bigg]+
 \end{equation*}
 \begin{equation*}
     +\mathbb{E}_{\mathbb{P}}[b(t) X_t \ln (Y_t) F_x(X_t,Y_t)-b(t) Y_t \ln (X_t) F_y(X_t,Y_t)]+
  \end{equation*}
 \begin{equation*}   
          -\mathbb{E}_{\mathbb{P}}[b(t) X_t \ln (Y_t) F_x(X_0,Y_0)-b(t) Y_t \ln (X_t) F_y(X_0,Y_0)]=0.   
 \end{equation*}

\bibliographystyle{plain}
\bibliography{references}

\end{document}